\documentclass[letterpaper,11pt]{article}
\usepackage[utf8]{inputenc}
\usepackage[in]{fullpage}
\usepackage{amsmath,amsthm,thmtools,amssymb}
\usepackage{color}
\usepackage{hyperref}
\usepackage{cleveref}
\usepackage{graphicx}
\usepackage{xspace}
\usepackage{comment}

\graphicspath{{././figure/}}%

\declaretheorem[style=plain,numberwithin=section,name=Theorem]{theorem}
\declaretheorem[style=plain,sibling=theorem,name=Lemma]{lemma}

\declaretheorem[style=plain,sibling=theorem,name=Corollary]{corollary}
\declaretheorem[style=plain,numbered=no,name=Claim]{claim}
\declaretheorem[style=plain,numbered=no,name=Subclaim]{subclaim}

\declaretheorem[style=definition,sibling=theorem,name=Example,qed=$\blacksquare$]{example}

\declaretheorem[style=definition,sibling=theorem,name=Remark,qed=$\blacksquare$]{remark}

\newcommand{\Ker}{\operatorname{Ker}}
\renewcommand{\Im}{\operatorname{Im}}
\newcommand{\Lk}{\operatorname{Lk}}
\newcommand{\St}{\operatorname{St}}

\newcommand{\NS}{\operatorname{NS}}
\newcommand{\vol}{\operatorname{vol}}

\newcommand{\Z}{\mathbb{Z}}
\newcommand{\R}{\mathcal{R}}
\newcommand{\C}{\mathcal{C}}
\newcommand{\calL}{\mathcal{L}}
\newcommand{\F}{\mathcal{F}}

\newcommand{\kcolpath}{{\sc $k$-Recoloring}\xspace}

\newcommand{\fourrecolor}{{\sc $4$-Recoloring}\xspace}
\newcommand{\fiverecolor}{{\sc $5$-Recoloring}\xspace}
\newcommand{\sixrecolor}{{\sc $6$-Recoloring}\xspace}
\newcommand{\krecolor}{{\sc $k$-Recoloring}\xspace}
\newcommand{\kplusrecolor}{{\sc $(k+1)$-Recoloring}\xspace}

\newcommand{\Lrecolor}{{\sc List-Recoloring}\xspace}
\newcommand{\conn}{{\sc Connectedness of $k$-Coloring Reconfiguration Graph}\xspace}
\newcommand{\fourconn}{{\sc Connectedness of $4$-Coloring Reconfiguration Graph}\xspace}

\title{Reconfiguration of colorings in triangulations of the sphere}
\author{
Takehiro Ito%
\thanks{Graduate School of Information Sciences, Tohoku University, Japan. {\tt takehiro@tohoku.ac.jp}}
\and 
Yuni Iwamasa%
\thanks{Graduate School of Informatics, Kyoto University, Japan. {\tt iwamasa@i.kyoto-u.ac.jp}}
\and
Yusuke Kobayashi%
\thanks{Research Institute for Mathematical Sciences, Kyoto University, Japan. {\tt yusuke@kurims.kyoto-u.ac.jp}}
\and 
Shun-ichi Maezawa%
\thanks{Department of Mathematics, Tokyo University of Science, Japan. {\tt maezawa.mw@gmail.com}}
\and
Yuta Nozaki%
\thanks{Graduate School of Advanced Science and Engineering, Hiroshima University, Japan. {\tt nozakiy@hiroshima-u.ac.jp}}
\and 
Yoshio Okamoto%
\thanks{Graduate School of Informatics and Engineering, The University of Electro-Communications, Japan. {\tt okamotoy@uec.ac.jp}}
\and
Kenta Ozeki%
\thanks{Faculty of Environment and Information Sciences, Yokohama National University, Japan. {\tt ozeki-kenta-xr@ynu.ac.jp}}
}
\date{\today}

\begin{document}

\maketitle

\begin{abstract}
In 1973, Fisk proved that
any $4$-coloring of a $3$-colorable triangulation of the $2$-sphere
can be obtained from any $3$-coloring by a sequence of Kempe-changes.
On the other hand,
in the case where we are only allowed to recolor a single vertex in each step,
which is a special case of a Kempe-change,
there exists a $4$-coloring that cannot be obtained from any $3$-coloring.
In this paper,
we present a
characterization of a $4$-coloring of a $3$-colorable triangulation of the $2$-sphere
that can be obtained from a $3$-coloring by a sequence of recoloring operations at single vertices,
and a
criterion for a $3$-colorable triangulation
of the $2$-sphere that all $4$-colorings can be obtained from a $3$-coloring by such a sequence.
Moreover, our first result can be generalized to a high-dimensional case,
in which ``$4$-coloring,'' ``$3$-colorable,'' and ``$2$-sphere'' above
are replaced with ``$k$-coloring,'' ``$(k-1)$-colorable,'' and ``$(k-2)$-sphere'' for $k \geq 4$,
respectively.
In addition, we show that the problem of deciding whether, for given two $(k+1)$-colorings,
one can be obtained from the other by such a sequence
is PSPACE-complete for any fixed $k \geq 4$.
Our results above can be rephrased as new results on the computational problems named {\sc $k$-Recoloring} and {\sc Connectedness of $k$-Coloring Reconfiguration Graph},
which are fundamental problems in the field of combinatorial reconfiguration.
\end{abstract}

\section{Introduction}\label{sec:intro}
In 1973, Fisk~\cite{Fis73I} proved that
all $4$-colorings of a $3$-colorable triangulation of the $2$-sphere are \emph{Kempe-equivalent},
that is,
for
any two $4$-colorings of the graph,
one is obtained from the other by a sequence of \emph{Kempe-changes}.
The method of Kempe-changes is known as a powerful tool 
for coloring of graphs (see e.g.,~\cite{Heawood,CR15}),
and has been intensively studied in graph theory (see e.g.,~\cite{Meyniel,LM,Mohar,MoharSalas2009, FJP,BBFJ,Feghali,BDL}).
In particular,
Mohar~\cite{Mohar} proved that
all $4$-colorings of a $3$-colorable planar graph are Kempe-equivalent
using Fisk's result,
and then Feghali~\cite{Feghali} improved this 
for $4$-critical planar graphs.
Mohar and Salas \cite{MoharSalas2009} extended Fisk's result
to toroidal triangulations.
As in those researches, 
Fisk's result is a fundamental one
in this context.

The formal definitions of Kempe-change and Kempe-equivalence are given as follows.
Let $\alpha \colon V(G) \to \{0,1,\dots,k-1\}$ be a $k$-coloring of a graph $G$,
let $a, b$ be two distinct colors in $\{0,1,\dots,k-1\}$,
and let $C$ be a connected component of 
the subgraph of $G$ induced by the vertices
colored with either $a$ or $b$.
Then, a \emph{Kempe-change} of $\alpha$ (at $C$)
is an operation 
to give rise to a new $k$-coloring
by exchanging the colors $a$ and $b$ on all vertices in $C$.
In particular, if $C$ consists of a single vertex,
then we refer to such a Kempe-change at $C$ as a \emph{single-change}.
Two $k$-colorings
of $G$
are \emph{Kempe-equivalent}
if one is obtained from the other 
by a sequence of Kempe-changes,
and \emph{single-equivalent}
if one is obtained from the other 
by a sequence of single-changes.

Let us return to Fisk's result for the Kempe-equivalence.
Let $G$ be a $3$-colorable triangulation of the 2-sphere.
The proof consists of the following two statements: All $3$-colorings of $G$ are Kempe-equivalent
under $4$-colorings,
and any two $4$-coloring of $G$ is Kempe-equivalent to a $3$-coloring.
Here, a $3$-coloring means that a coloring uses only three colors in $\{0,1,2,3\}$.
The first statement, which is a folklore, can be easily obtained as follows.
Since $G$ is a $3$-colorable triangulation of the 2-sphere,
for any two $3$-colorings $\alpha, \beta$ of $G$
there uniquely exists a permutation $\pi$ on $\{0,1,2,3\}$ such that $\beta = \pi \circ \alpha$.
Then, according to $\pi$, we can obtain $\beta$ from $\alpha$ by a sequence of Kempe-changes (under $4$-colorings) each of which changes a color at only one vertex, namely, a sequence of single-changes, by using the fourth color not appearing in $\alpha$.
Therefore, the nontrivial and crucial part in Fisk's result is to show the second statement.

The above observation for the first statement says that
all $3$-colorings of $G$ are single-equivalent
under $4$-colorings.
On the other hand, in general, some $4$-coloring is not single-equivalent to any of $3$-colorings; see Figure~\ref{fig:unbalanced 4-coloring} for example.
\begin{figure}
\centering
\includegraphics{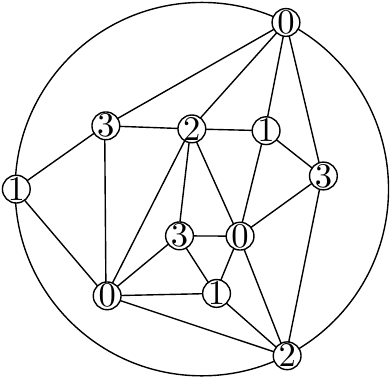}
\caption{A $4$-coloring of a $3$-colorable triangulation of the 2-sphere
such that it is not single-equivalent to any of $3$-colorings; no vertex can be recolored by a single-change.}
\label{fig:unbalanced 4-coloring}
\end{figure}
Here natural questions arise: \emph{What $4$-colorings are single-equivalent to some $3$-coloring?}~and \emph{in what $3$-colorable triangulation of the $2$-sphere all $4$-colorings are single-equivalent?}

In this paper, we resolve these questions in the following sense.
\begin{enumerate}
    \item We present a characterization
    for a $4$-coloring of $G$
to be single-equivalent to some $3$-coloring (\Cref{thm:main}).
In addition, we show that, for any $4$-colorings $\alpha, \beta$ of $G$ single-equivalent to some $3$-coloring,
there exists a sequence of single-changes of length $O(\#V(G)^2)$ from $\alpha$ to $\beta$ (\Cref{thm:diameter}).
    \item We provide a criterion
for a $3$-colorable triangulation of the $2$-sphere
that all $4$-colorings are single-equivalent (\Cref{thm:conn}).
\end{enumerate}
Furthermore, we consider a triangulation of a high-dimensional sphere.
Let $G$ be a $(k-1)$-colorable triangulation of the $(k-2)$-sphere for some positive integer $k \geq 4$.
Then, by the same argument as in the case of $k = 4$ above,
all $(k-1)$-colorings of $G$ are single-equivalent under $k$-colorings.
The following is a generalization of our first results (\Cref{thm:main} and \Cref{thm:diameter}):
\begin{enumerate}
\setcounter{enumi}{2}
    \item
    We present a characterization
for a $k$-coloring of a $(k-1)$-colorable triangulation $G$ of the $(k-2)$-sphere
to be single-equivalent to some $(k-1)$-coloring (\Cref{thm:highdim_case}).
In addition, we show that, for any $k$-colorings $\alpha, \beta$ of $G$ single-equivalent to some $(k-1)$-coloring,
there exists a sequence of single-changes of length $O(\#V(G)^{2\lfloor (k-1)/2 \rfloor})$ from $\alpha$ to $\beta$ (\Cref{thm:diam_high}).
\end{enumerate}
In fact, the third result can be further generalized to $(k-1)$-colorable triangulations of connected closed $(k-2)$-manifolds satisfying a certain condition.

Our results are deeply related to the computational problems named \krecolor and \conn,
which ask the connectedness of a $k$-coloring reconfiguration graph.
Here, the \emph{$k$-coloring reconfiguration graph} of a $k$-colorable graph $G$,
denoted by $\R_k(G)$, is a graph such that
its vertex set consists of all $k$-colorings of $G$
and there is an edge between two $k$-colorings $\alpha$ and $\beta$ of $G$
if and only if $\beta$ is obtained from $\alpha$ by recoloring only a single vertex in $G$, i.e., by a single-change.
Thus,
two $k$-colorings of $G$ are single-equivalent
if and only if
they are connected in $\R_k(G)$.
Then \krecolor and \conn are defined as follows.
\begin{description}
        \item[\underline{\krecolor}]
    	\item[Input:] A $k$-colorable graph $G$ and $k$-colorings $\alpha$ and $\beta$ of $G$.
    	\item[Output:] YES if $\alpha$ and $\beta$ are connected in $\R_k(G)$, and NO otherwise.
\end{description}
\begin{description}
        \item[\underline{\conn}]
    	\item[Input:] A $k$-colorable graph $G$.
    	\item[Output:] YES if $\R_k(G)$ is connected, and NO otherwise.
\end{description}

The problems \krecolor and \conn are fundamental in the recently emerging field of \emph{combinatorial reconfiguration} (see \cite{Heuvel,Nishimura} for surveys),
which are extensively studied.
It is shown that
\kcolpath is polynomial-time solvable if $k \leq 3$ \cite{CHJ11},
while PSPACE-complete if $k \geq 4$ \cite{CHJ}.
According to \cite[Section 3.2]{Heuvel},
the situation is very different from that for Kempe-equivalence,
whose complexity is widely open.
Bonsma and Cereceda~\cite{bonsma} considered \krecolor for (bipartite) planar graphs;
\krecolor for planar graphs is PSPACE-complete if $4 \leq k \leq 6$
and that for bipartite planar graphs is PSPACE-complete if $k = 4$.
Cereceda, van den Heuvel, and Johnson~\cite{CHJ} showed that $\mathcal{R}_k(G)$ is connected
for any $(k-2)$-degenerate graph~\cite{CHJ}.
By combining it with the fact that any planar graph is $5$-degenerate
and any bipartite planar graph is $3$-degenerate,
we see that 
\kcolpath and \conn are in P (all instances are YES-instances) for any planar graph with $k \geq 7$
and for any bipartite planar graph with $k \geq 5$.
In another paper~\cite{CHJ09}, Cereceda, van den Heuvel, and Johnson also showed that
{\sc Connectedness of $3$-Coloring Reconfiguration Graph} is coNP-complete in general
and is in P for bipartite planar graphs.

The problem \conn is also fundamental in the studies of the Glauber dynamics (a class of Markov chains) for $k$-colorings of a graph, which are used for random sampling and approximate counting.
In each step of the Glauber dynamics of $k$-colorings, we are given a $k$-coloring of a graph.
Then, we pick a vertex $v$ and a color $c$ uniformly at random, and change the color of $v$ to $c$ when the neighbors of $v$ are not colored by $c$.
Hence, one step of this Markov chain is exactly a single-exchange as long as we move to another coloring, and the state space is identical to the $k$-coloring reconfiguration graph.
The connectedness of the $k$-coloring reconfiguration graph ensures the Markov chain to be irreducible.
For the Glauber dynamics, the mixing property is one of the main concerns.
It is an open question whether the Glauber dynamics of $k$-colorings has polynomial mixing time when $k \geq \Delta + 2$, where $\Delta$ is the maximum degree of a graph \cite{DBLP:journals/rsa/Jerrum95}.
From continuing work in the literature, we know that the Glauber dynamics mixes fast when $k > 2\Delta$ \cite{DBLP:journals/rsa/Jerrum95}, 
$k > \frac{6}{11}\Delta$ \cite{doi:10.1063/1.533196}, and finally $k > (\frac{6}{11}-\varepsilon)\Delta$ for a small absolute constant $\varepsilon > 0$ \cite{DBLP:conf/soda/ChenDMPP19}.
Results on restricted classes of graphs have also been known.
For example, Hayes, Vera and Vigoda~\cite{DBLP:journals/rsa/HayesVV15} proved that the Glauber dynamics mixes fast for planar graphs when $k = \Omega(\Delta/\log \Delta)$.

Our proofs provide algorithms for special cases of \krecolor and \conn.
Here, we are supposed to be given a simplical complex $K$ whose geometric realization is homeomorphic to the $(k-2)$-sphere such that its $1$-skeleton $G$ is $(k-1)$-colorable.
As we have seen, all $(k-1)$-colorings of $G$ belong to the same connected component of $\mathcal{R}_k(G)$;
we refer to it as the \emph{$(k-1)$-coloring component} of $\mathcal{R}_k(G)$.
Our third result (including the first) implies that,
provided one of the input $k$-colorings $\alpha$ and $\beta$ belongs to the $(k-1)$-coloring component of $\mathcal{R}_k(G)$,
the problem \krecolor for $G$
can be solved in linear time in the size $\# K$ of the input simplical complex $K$.
In particular, if $k$ is fixed,
then our result says that it can be solved in polynomial time in $\# V(G)$.
Our second result implies that
\fourconn for a $3$-colorable triangulation of the $2$-sphere
can be solved in linear time in $\# V(G)$.

We further investigate the computational complexity of the recoloring problem for a $(k-1)$-colorable triangulation $G$ of the $(k-2)$-sphere.
It is still open whether \krecolor for $G$ can be solved in polynomial time,
although we prove the polynomial-time solvability of the special case where one of the input $k$-colorings $\alpha$ and $\beta$ belongs to the $(k-1)$-coloring component of $\mathcal{R}_k(G)$.
In this paper, we additionally show that, if the number of colors which we can use increases by one,
then it is difficult to check the single-equivalence between given two colorings:
\begin{enumerate}
\setcounter{enumi}{3}
\item For any fixed $k \geq 4$, the problem \kplusrecolor is PSPACE-complete
for $(k-1)$-colorable triangulations of the $(k-2)$-sphere (\Cref{thm:pspace}).
\end{enumerate}
In the case of $k=4$,
our result is stronger than
the PSPACE-completeness of \fiverecolor for planar graphs, which is known in the literature~\cite{bonsma}.

We here emphasize that, for our algorithmic results, we are given a triangulation of a sphere, but not only its $1$-skeleton.
This assumption is justified by the following reasons.
For each fixed $d \geq 5$, the sphere recognition problem is undecidable \cite{VKF,CHERNAVSKY2006325}: Namely it is undecidable whether a given simplicial complex is a triangulation of the $d$-sphere.
This implies that it is also undecidable whether a given graph is the $1$-skeleton of some triangulation of the $d$-sphere.
When $d=3$, the sphere recognition is decidable \cite{10.1007/978-3-0348-9078-6_54,thompson}, but not known to be solved in polynomial time (while it is known to be in NP \cite{schleimer});
the decidability is open when $d=4$.
Therefore, when $d \geq 3$, to filter out the intrinsic intractability of sphere recognition, we assume a triangulation is also given along with a graph. 
On the other hand, when $d=2$, we can decide whether a graph is the $1$-skeleton of some triangulation in linear time~\cite{DBLP:journals/jacm/HopcroftT74}.
In this case, the size of a triangulation is the same as the size of its $1$-skeleton in the order of magnitude by Euler's formula, and therefore, the assumption that a triangulation is also given is not relevant.

\paragraph{Organization.}
This paper is organized as follows.
In \Cref{sec:preliminaries},
we give several notations.
We provide a characterization on the $3$-coloring component of a $3$-colorable triangulation of the $2$-sphere in \Cref{subsec:2dim_case},
which answers the first question.
\Cref{subsec:highdim_case} deals with its high-dimensional generalization.
\Cref{sec:connR} is devoted to resolving the second question:
We present a criterion
for a $3$-colorable triangulation of the $2$-sphere
that any two $4$-colorings are single-equivalent in \Cref{sec:connR}.
In \Cref{sec:pspace}, we show the PSPACE-completeness of \kplusrecolor
for $(k-1)$-colorable triangulations of the $(k-2)$-sphere for $k \geq 4$.
\Cref{sec:conclusion} concludes this paper with several open questions.

\section{Preliminaries}
\label{sec:preliminaries}

For a set $A$,
we denote by $\# A$ the cardinality of $A$.

For a graph $G$,
its vertex set and edge set are denoted by $V(G)$ and $E(G)$, respectively.
For $v \in V(G)$,
we denote by $N_G(v)$ the set of neighbors of $v$
and by $\delta_G(v)$ the set of edges incident to $v$;
we simply write $N(v)$ and $\delta(v)$ if $G$ is clear from the context.
A map $\alpha \colon V(G) \to \{0,1,\dots,k-1\}$ is called a \emph{$k$-coloring}
if $\alpha(u) \neq \alpha(v)$ for each edge $\{u,v\} \in E(G)$.
A vertex $v \in V(G)$ is said to be \emph{recolorable} with respect to a $k$-coloring $\alpha$
if there is a $k$-coloring $\alpha'$ such that $\alpha'(u) = \alpha(u)$ for $u \in V \setminus \{v\}$ and $\alpha'(v) \neq \alpha(v)$, i.e.,
we can change the color $\alpha(v)$ of $v$.

Let $S^d$ denote the $d$-sphere.
A \emph{triangulation} of $S^d$ is a pair of a simplicial complex $K$ and a homeomorphism $h\colon |K|\to S^d$, where $|K|$ denotes the geometric realization of $K$.
See, for instance, Munkres~\cite{Mun84} for fundamental terminology in simplicial complexes.
Throughout this paper, we identify $|K|$ with $S^d$ and omit to write $h$.
For a simplex $\sigma \in K$, its \emph{star complex} $\St_K(\sigma)$ and \emph{link complex} $\Lk_K(\sigma)$ are defined by
\begin{align*}
    \St_K(\sigma)&:=\{\tau\in K\mid \text{$\sigma$ and $\tau$ are faces of a common simplex in $K$}\}, \\
    \Lk_K(\sigma)&:=\{\tau\in K\mid \sigma\cap\tau=\emptyset,\ \sigma\ast\tau \in K\},
\end{align*}
where $\sigma\ast\tau$ denotes the join of $\sigma$ and $\tau$ (see \cite[Section~62]{Mun84}).
Figure~\ref{fig:starlink_example1} shows examples.
\begin{figure}
\centering
\includegraphics[width=\textwidth]{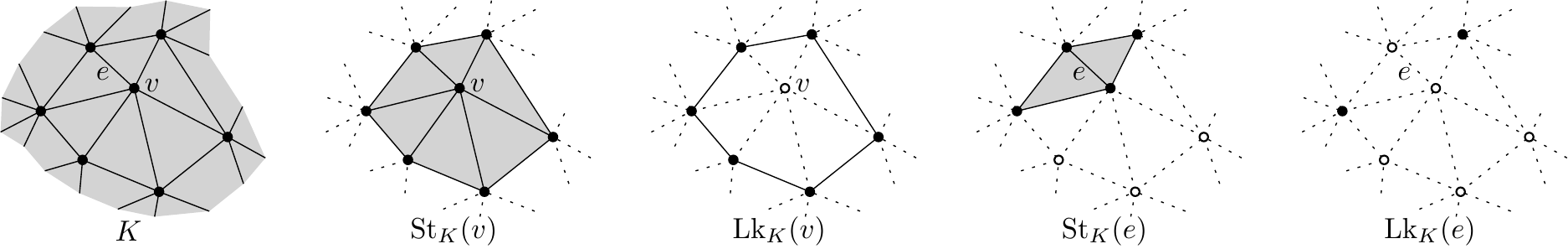}
\caption{An example of the star complexes and the link complexes of a $2$-dimensional simplicial complex $K$.}
\label{fig:starlink_example1}
\end{figure}
Also, let $\St^d_K(\sigma)$ denote the $d$-simplices in $\St_K(\sigma)$.
For a subset $K' \subseteq K$, we define $|K'|\subseteq S^d$ by $|K'|:=\bigcup_{\sigma\in K'}\sigma$.
For instance, if $v$ is a vertex of a triangulation of a surface without boundary, then $|\St_K(v)|$ and $|\Lk_K(v)|$ are homeomorphic to a closed disk and a circle, respectively.
In this paper, we specify a triangulation by an embedded graph $G$ in $S^d$, which is actually the $1$-skeleton of a triangulation $K$.
Also, we suppose that the input of \krecolor and \conn is the simplical complex $K$;
for example, we are given the set of faces of a triangulation of the $2$-sphere.
We use $\St_G(\sigma)$ instead of $\St_K(\sigma)$ by abuse of notation.
For example, $\St^0_G(v)\setminus\{v\}=N_G(v)$ and $\St^1_G(v)\setminus\Lk_G(v)=\delta_G(v)$.
Also, we simply write $\St(\sigma)$ and $\Lk(\sigma)$ if $G$ or $K$ is clear from the context.

It is well-known that 
a triangulation of the $2$-sphere
is $3$-colorable if and only if 
every vertex has an even degree (i.e., Eulerian).
In this sense, a $3$-colorable triangulation is said to be \emph{even}.
More generally, a triangulation $K$ of a closed $d$-manifold is \emph{even} if $\#\St^{d}(\sigma^{d-2})$ is even for every $(d-2)$-simplex $\sigma^{d-2}\in K$, where $d\geq 2$.
If the $1$-skeleton of $K$ is $(d+1)$-colorable, then $K$ is even.
By \cite[Sections~I.4 and VI.2]{Fis77}, the converse is also true for $S^{d}$, more generally, for simply-connected manifolds.
Hence, it is easy to check whether a given triangulation of $S^{d}$ is $(d+1)$-colorable.

\section{A characterization on the $(k-1)$-coloring component}
\label{sec:characterization}
In this section,
we resolve (a generalization of) the first question posed in Introduction: \emph{In a $(k-1)$-colorable triangulation $G$ of the $(k-2)$-sphere, what $k$-colorings are single-equivalent to some $(k-1)$-coloring?}
In \Cref{subsec:2dim_case},
we consider the two-dimensional case, i.e., $k = 4$;
we present a characterization for a $4$-coloring of a $3$-colorable triangulation $G$ of the $2$-sphere
to be single-equivalent to some $3$-coloring.
A characterization for high-dimensional cases ($k \geq 4$) can be obtained by a similar argument,
which is given in \Cref{subsec:highdim_case}.

Recall that all $(k-1)$-colorings of a $(k-1)$-colorable triangulation $G$ of the $(k-2)$-sphere belong to the same connected component of $\mathcal{R}_k(G)$;
we refer to it as the $(k-1)$-coloring component of $\mathcal{R}_k(G)$.
In this section, for a graph $G$,
its vertex set and edge set are simply denoted by $V$ and $E$, respectively.
For sets $A$ and $B$,
let $A \bigtriangleup B$ denote the symmetric difference $(A \setminus B) \cup (B \setminus A)$ of $A$ and $B$.
A set family $\F \subseteq 2^A$ is said to be \emph{laminar}
if, for any $X, Y \in \F$,
we have $X \subseteq Y$, $X \supseteq Y$, or $X \cap Y = \emptyset$.

\subsection{Two-dimensional case}
\label{subsec:2dim_case}
Let $G$ be a $3$-colorable triangulation of the $2$-sphere
and $F$ the set of faces of $G$.
We first define the signature on a face in $F$ with respect to a $4$-coloring of $G$
and its related concepts,
which were originally introduced in \cite{heawood1898four} (see also \cite[Section 8 of Chapter 2]{Saaty72}).
These play an important role in our characterization.

Let $\alpha \colon V \to \{0,1,2,3\}$ be a $4$-coloring of
$G$.
We assign a signature $+1$/$-1$ to each face $f \in F$ so that, for every pair of adjacent faces $f, f'$ with $f = \{ u, v, w \}$ and $f' = \{u', v, w\}$,
they have the same signature if and only if $\alpha(u) \neq \alpha(u')$,
where two faces $f, f' \in F$ are said to be \emph{adjacent}
if $f$ and $f'$ share an edge, i.e., $\#(f \cap f') = 2$.
Such an assignment can be obtained as follows.
For each face $f = \{u,v,w\} \in F$,
we denote by $[\alpha(f)]$ the cyclically ordered set $[\alpha(u)\alpha(v)\alpha(w)]$ on $\{\alpha(u), \alpha(v), \alpha(w)\}$,
where $u, v, w$ are arranged in counterclockwise order in $G$ if we see it from the outside of the $2$-sphere.
We define $\varepsilon_\alpha \colon F \to \{+1, -1\}$ by
\begin{align*}
    \varepsilon_\alpha(f) :=
    \begin{cases}
        +1 & \text{if $[\alpha(f)] \in \{ [123], - [023], [013], - [012]\}$},\\
        -1 & \text{if $[\alpha(f)] \in \{ -[123], [023], -[013], [012]\}$},
    \end{cases}
\end{align*}
where the minus sign $-$ indicates the opposite order, that is, $-[ijk]=[jik]$.
We note here that when we regard $[123], - [023], [013], - [012]$ as oriented $2$-simplices, they appear in the boundary of an oriented $3$-simplex $[0123]$: $\partial [0123] = [123] \cup - [023] \cup [013] \cup -[012]$.
A face $f \in F$ with $\varepsilon_\alpha(f) = +1$ (resp.\  $\varepsilon_\alpha(f) = -1$)
is called a \emph{$+$-face} (resp.\  \emph{$-$-face}) with respect to $\alpha$.
Figure~\ref{fig:signassignment_example1} shows an example.
\begin{figure}
\centering
\includegraphics{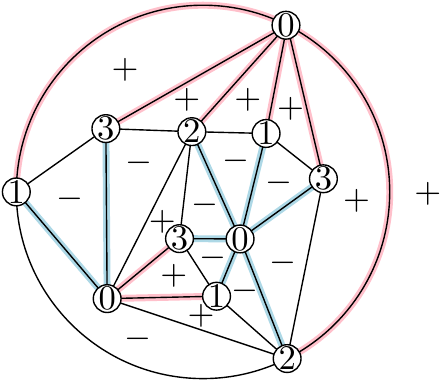}
\caption{An example of the signature assignment to the faces. Red edges depict the $+$-nonsingular edges and blue edges depict the $-$-nonsingular edges.}
\label{fig:signassignment_example1}
\end{figure}
Recall that, for $v \in V$, the set of faces containing $v$ is denoted as $\St^2(v)$.
For a $4$-coloring $\alpha$, let $\varepsilon_{\alpha, v}^{-1}(+1)$ (resp.\ $\varepsilon_{\alpha, v}^{-1}(-1)$) denote the set of $+$-faces (resp.\ $-$-faces) in $\St^2(v)$.

An edge $e \in E$ is said to be \emph{singular} with respect to $\alpha$
if the two adjacent faces $f, f' \in F$ sharing $e$ have different signatures, i.e.,
$\varepsilon_\alpha(f) \neq \varepsilon_\alpha(f')$,
and to be \emph{nonsingular}
if it is not singular.
A nonsingular edge is particularly said to be \emph{$+$-nonsingular} (resp.\ \emph{$-$-nonsingular})
if $\varepsilon_\alpha(f) = \varepsilon_\alpha(f') = +1$ (resp.\ $\varepsilon_\alpha(f) = \varepsilon_\alpha(f') = -1$).
Figure~\ref{fig:signassignment_example1} also illustrates the $+$- and $-$-nonsingular edges.
For $v \in V$,
we denote by $\NS_\alpha(v)$, $\NS_\alpha^+(v)$, and $\NS_\alpha^-(v)$ the set of nonsingular, $+$-nonsingular, and $-$-nonsingular edges incident to $v$,
respectively.
Also the set of nonsingular edges is denoted as $\NS_\alpha$.
The following are obtained by direct observations.
\begin{lemma}
\label{lemma:3-coloring}
Let $\alpha$ be any $4$-coloring of a $3$-colorable triangulation $G$ of the $2$-sphere.
\begin{itemize}
    \item[{\rm (1)}]
    A vertex $v \in V$ is recolorable with respect to $\alpha$ if and only if
all edges incident to $v$ are singular, i.e., $\NS_\alpha(v) = \emptyset$.
    \item[{\rm (2)}]
    The coloring $\alpha$ is a $3$-coloring if and only if all edges are singular, i.e., $\NS_\alpha = \emptyset$.
\end{itemize}
\end{lemma}

We can derive a necessary condition for a $4$-coloring $\alpha$ of $G$ to belong to the $3$-coloring component of $\mathcal{R}_4(G)$ as follows.
Let $\alpha'$ be a $4$-coloring obtained from $\alpha$ by changing the color of $v$, i.e., $\alpha'(v) \neq \alpha(v)$ and $\alpha'(u) = \alpha(u)$ for all $u \in V \setminus \{v\}$.
Then the signatures of all faces in $\St^2(v)$ are inverted (see also Figure~\ref{fig:recolor}):
\begin{figure}
\centering
\includegraphics{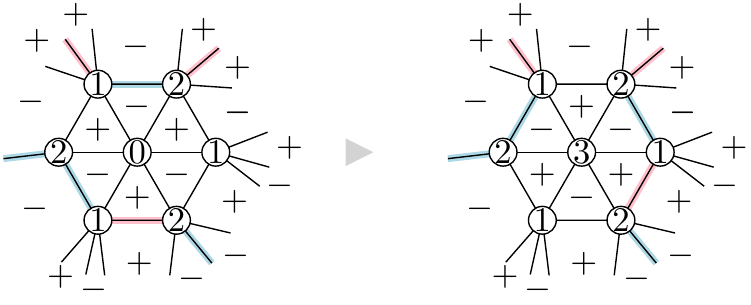}
\caption{An example of the change of the signatures by a single-change.
As in Figure~\ref{fig:signassignment_example1}, red and blue edges depict the $+$- and $-$-nonsingular edges, respectively.}
\label{fig:recolor}
\end{figure}
\begin{align}\label{eq:epsilon}
    \varepsilon_{\alpha'}(f) =
    \begin{cases}
        -\varepsilon_\alpha(f) & \text{if $f \in \St^2(v)$},\\
        \varepsilon_\alpha(f) & \text{if $f \notin \St^2(v)$}.
    \end{cases}
\end{align}
This implies that,
if $\alpha$ and $\alpha'$ belong to the same connected component in $\mathcal{R}_4(G)$,
then we have
\begin{align}\label{eq:necessary}
    \#\varepsilon_{\alpha, v}^{-1}(+1) = \#\varepsilon_{\alpha', v}^{-1}(+1) \quad
    \text{and} \quad
    \#\varepsilon_{\alpha, v}^{-1}(-1) = \#\varepsilon_{\alpha', v}^{-1}(-1)
\end{align}
for all $v \in V$.
Furthermore,
\Cref{lemma:3-coloring}~(2) implies that,
if $\alpha$ is a $3$-coloring of $G$,
then we have $\varepsilon_\alpha(f) \neq \varepsilon_\alpha(f')$ for each $v \in V$ and every adjacent $f, f' \in \St^2(v)$.
Therefore,
it follows from the equation~\eqref{eq:necessary} and \Cref{lemma:3-coloring}~(2) that the following \emph{balanced condition} holds
if $\alpha$ belongs to the $3$-coloring component of $\R_4(G)$:
\begin{description}
    \item[{\rm (B)}]
    For each $v \in V$,
    \begin{align}\label{eq:iff-condition}
        \#\varepsilon_{\alpha, v}^{-1}(+1) = \#\varepsilon_{\alpha, v}^{-1}(-1).
    \end{align}
\end{description}

Our main result in this subsection is showing that the balanced condition (B) is also a sufficient condition,
that is, the condition (B) characterizes the $3$-coloring component of $\mathcal{R}_4(G)$.
\begin{theorem}
\label{thm:main}
Let $\alpha \colon V \to \{0,1,2,3\}$ be a $4$-coloring of a $3$-colorable triangulation $G$ of the $2$-sphere.
Then, $\alpha$ belongs to the $3$-coloring component of $\mathcal{R}_4(G)$ if and only if
it satisfies
the balanced condition {\rm (B)}.
\end{theorem}

For the proof of \Cref{thm:main},
we observe the behavior of $\NS_{\alpha}$ when we recolor a vertex
from a $4$-coloring $\alpha$.
If we change the color $\alpha(v)$ of a vertex $v$,
then 
it follows from the equation~\eqref{eq:epsilon} that
all singular edges in $\Lk(v)$ will be nonsingular
and vice versa (see Figure~\ref{fig:recolor}).
Thus, the following holds.
\begin{lemma}\label{lem:nonsingular}
Let 
$\alpha$ be a $4$-coloring of a $3$-colorable triangulation $G$ of the $2$-sphere and $\alpha'$ a $4$-coloring obtained from $\alpha$ by changing the color of a vertex $v$.
Then
\begin{align*}
    \NS_{\alpha'}(u) =
    \begin{cases}
        \NS_\alpha(u) & \text{if $u \notin N(v)$},\\
        \NS_\alpha(u) \bigtriangleup \left(\Lk(v) \cap \delta(u)\right) & \text{if $u \in N(v)$}.
    \end{cases}
\end{align*}
In particular, $\NS_{\alpha'} = \NS_{\alpha} \bigtriangleup \Lk(v)$.
\end{lemma}

In our proof, the set $\NS_\alpha$ of nonsigular edges is viewed as the disjoint union of closed trails in $G$.
Here, a \emph{closed trail} is a closed walk such that all edges are distinct.
For a closed trail $C$ of $G$ and a vertex $v \in V$,
we denote by $C_v$ the set of subpaths of $C$ obtained from the restriction of $C$ to $\delta(v)$, i.e.,
$C_v := \{ \{e, e'\} \mid \text{$\{e, e'\}$ is a subpath of $C$ such that $e,e' \in \delta(v)$} \}$.
A closed trail $C$ of $G$ is said to be \emph{noncrossing}
if for any vertex $v$, no pair of subpaths $P, P' \in C_v$ crosses in $S^2$, i.e.,
$P'$ is contained in the closure of a connected component of $|\St(v)| \setminus P$ in $S^2$,
where $P$ is viewed as a curve in $S^2$.
For a noncrossing closed trail $C$ with fixed orientation,
we define ${\rm L}_C$ by the union of connected components of $S^2 \setminus C$ such that it lies in the left side of some edge in $C$.
Similarly, we define ${\rm R}_C$ by the union of connected components of $S^2 \setminus C$ such that it lies in the right side of some edge in $C$.
Since $C$ is noncrossing,
the family $\{ {\rm L}_C, {\rm R}_C \}$ forms a bipartition of $S^2 \setminus C$.

We fix a certain face $f_{\rm out} \in F$ as the \emph{outer face} of $G$.
We say that one of ${\rm L}_C$ and ${\rm R}_C$ is the \emph{outside} of $C$
if it contains the outer face $f_{\rm out}$.
The other is called the \emph{inside} of $C$.
Let $F_C \subseteq F$ denote the set of faces in the inside of $C$.
For a set $\C$ of noncrossing closed trails in $G$,
we define the \emph{volume} of $\C$, denoted as $\vol(\C)$, by
the sum of the number of faces contained in the inside of $C$ over all $C \in \C$, i.e.,
\begin{align*}
    \vol(\C) := \sum_{C \in \mathcal{C}}\# F_C.
\end{align*}

It is known that,
for any $4$-coloring $\alpha$ of $G$ and $v \in V$,
the number $\# \NS_\alpha(v)$ of nonsingular edges incident to $v$ is even (see e.g., \cite[Lemma 5]{Fis73I}).
For $v \in V$,
let $\pi_v$ be a partition of $\NS_\alpha(v)$ such that each member of $\pi_v$ is of size two (such a partition exists since $\# \NS_\alpha(v)$ is even),
and define $\pi := \bigcup_{v \in V} \pi_v$.
We refer to $\pi$ as an \emph{NS-pairing} (with respect to $\alpha$).
An NS-pairing $\pi = \bigcup_{v \in V} \pi_v$
uniquely determines a family $\C_\pi$ of closed trails in $G$
satisfying that all closed trails in $\C_\pi$ are disjoint
and $\pi_v = \bigcup_{C \in \C_\pi} C_v$ for all $v \in V$.
Note that $\NS_\alpha$ equals the disjoint union of all closed trails $C \in \C_\pi$.
\figurename~\ref{fig:ns_pairing1} provides an example of the set of closed trails induced by an NS-pairing.

\begin{figure}
    \centering
    \includegraphics[width=0.8\textwidth,page=2]{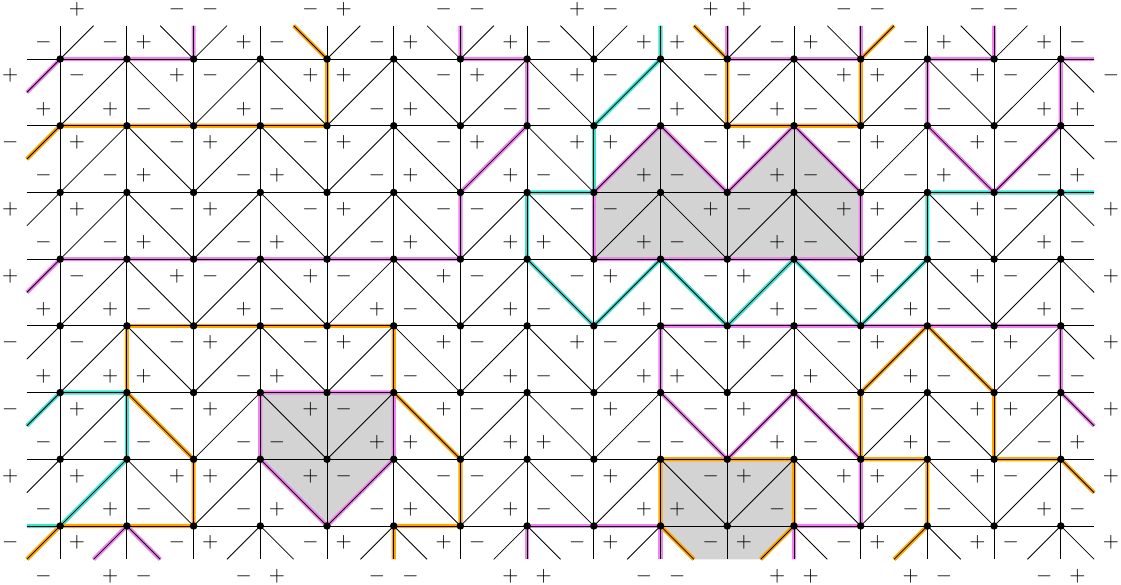}
    \caption{An example of NS-pairings. Colors show closed trails. Note that this NS-pairing is admissible. The gray areas show innermost closed trails.}
    \label{fig:ns_pairing1}
\end{figure}

An NS-pairing $\pi = \bigcup_{v \in V} \pi_v$ is said to be \emph{admissible}
if the following hold for any $v \in V$:
\begin{description}
    \item[{\rm (A1)}] All members of $\pi_v$ consist of one $+$-nonsingular edge and one $-$-nonsingular edge;
    \item[{\rm (A2)}] No pair $P, P' \in \pi_v$ crosses in $|\St(v)| \subseteq S^2$.
\end{description}
Let $\pi$ be an admissible NS-pairing.
Since each $C \in \C_\pi$ is noncrossing by (A2),
the inside of $C$, and hence $F_C$, are well-defined.
We define the face set family $\calL_\pi \subseteq 2^F$ by $\calL_\pi := \{ F_C \mid C \in \C_\pi \}$.

The admissibility of $\pi$ induces interesting properties on $\C_\pi$ and $\calL_\pi$ as follows.

\begin{lemma}\label{lem:admissible}
Let $\pi$ be an admissible NS-pairing with respect to a $4$-coloring $\alpha$.
\begin{itemize}
    \item[{\rm (1)}]
    The restriction of $\alpha$ to $C$ is a $2$-coloring.
    \item[{\rm (2)}]
    The family $\calL_\pi$ is laminar.
\end{itemize}
\end{lemma}

\begin{proof}
(1).
Take any member $\{\{u, v\}, \{v, w\}\}$ of $\pi_v$, which forms a subpath of some $C \in \C_\pi$.
It suffices to show that $\alpha(u) = \alpha(w)$.
We may assume that $\alpha(v) = 3$.

Let $n_+$ (resp. $n_-$) denote the number of $+$-faces (resp. $-$-faces) in $\St^2(v) \cap F_C$.
By the definition of the signature map $\varepsilon_\alpha$, we have
\begin{align*}
    \alpha(w) \equiv \alpha(u) + (n_+ - n_-) \pmod 3 \qquad \text{or} \qquad
    \alpha(w) \equiv \alpha(u) - (n_+ - n_-) \pmod 3.
\end{align*}
By the noncrossingness of $\pi_v$,
the set of nonsingular edges incident to $v$ in the inside of $C$ is of the form of the union of a subset of $\pi_v$.
Moreover, since all members of $\pi_v$ consist of one $+$-nonsingular edge and one $-$-nonsingular edge,
the number of $+$-nonsingular edges incident to $v$ in the inside of $C$
equals that of $-$-nonsingular edges.
This implies that $n_+ = n_-$.
Thus $\alpha(u) = \alpha(w)$ follows, as required.

(2).
Take any two closed trails $C, C' \in \C_\pi$.
Since $\pi$ is admissible, in particular, no pair of members in $\pi_v$ crosses in $|\St(v)|$ for any $v \in V$,
the closed trail $C'$ is contained in either the inside or the outside of $C$.
Thus, in the former case we have $F_{C'} \subseteq F_C$, and in the latter case we have $F_C \subseteq F_{C'}$ or $F_C \cap F_{C'} = \emptyset$,
which implies that $\calL_\pi$ is laminar.
\end{proof}

\Cref{lem:admissible}~(2) implies that $\C_\pi$ has an innermost closed trail in $S^2$, which corresponds to a minimal set in $\calL_\pi$.

We are ready to prove \Cref{thm:main}.
\begin{proof}[Proof of \Cref{thm:main}]
We have already seen the only-if part.
In the following, we show the if part.
Let $\alpha \colon V \to \{0,1,2,3\}$ be a $4$-coloring of $G$ satisfying the balanced condition (B) but not a $3$-coloring,
i.e., $\NS_\alpha \neq \emptyset$
by \Cref{lemma:3-coloring}~(2).

We first see that $\alpha$ has an admissible NS-pairing.
Since
\begin{align*}
    2 \cdot \#\varepsilon_{\alpha, v}^{-1}(+1) &= 2 \cdot \#\NS_\alpha^+(v) + \#(\delta(v) \setminus \NS_\alpha(v)) \quad \text{and}\\
    2 \cdot \#\varepsilon_{\alpha, v}^{-1}(-1) &= 2 \cdot \#\NS_\alpha^-(v) + \#(\delta(v) \setminus \NS_\alpha(v)),
\end{align*}
we have 
\[
    \# \NS_\alpha^+(v) = \# \NS_\alpha^-(v)
\]
by (B).
We construct an admissible NS-pairing as follows.
For $v \in V$,
let $\pi' := \emptyset$, $N_v^+ := \NS_\alpha^+(v)$, and $N_v^- := \NS_\alpha^-(v)$.
While $N_v^+ \neq \emptyset$ and $N_v^- \neq \emptyset$,
we take $e^+ \in N_v^+$ and $e^- \in N_v^-$ such that one of the connected components of $|\St(v)| \setminus \{ e^+, e^- \}$ contains no edges in $N_v^+ \cup N_v^-$
(such a pair $(e^+, e^-)$ always exists)
and update $\pi' \leftarrow \pi' \cup \{\{ e^+, e^- \}\}$, $N_v^+ \leftarrow N_v^+ \setminus \{ e^+ \}$,
and $N_v^- \leftarrow N_v^- \setminus \{ e^- \}$.
After the above procedure stops,
we define $\pi_v$ as the resulting $\pi'$.
Then, we can see that $\pi_v$ is satisfies (A1) and (A2).
Therefore, $\pi := \bigcup_{v \in V} \pi_v$ is an admissible NS-pairing.

The following claim is crucial for the proof of \Cref{thm:main}.
\begin{claim}
There exists a recolorable vertex $v_0 \in V$
such that the $4$-coloring $\alpha'$ obtained from $\alpha$ by recoloring $v_0$ has an admissible NS-pairing $\pi'$ satisfying $\vol(\C_{\pi'}) < \vol(\C_\pi)$.
\end{claim}
If this claim is true, then by recoloring such $v_0$ repeatedly,
we finally obtain a $4$-coloring $\alpha^*$ and an admissible NS-pairing $\pi^*$ with respect to $\alpha^*$ such that $\vol(\C_{\pi^*}) = 0$.
The equality $\vol(\C_{\pi^*}) = 0$ implies $\NS_{\alpha^*} = \emptyset$,
i.e., $\alpha^*$ is actually a $3$-coloring by \Cref{lemma:3-coloring}~(2).
Therefore,
$\alpha$ belongs to the $3$-coloring component of $\R_4(G)$,
as required.

In the following, we show the claim.
Take an arbitrary innermost closed trail $C \in \C_\pi$,
the existence of which is verified by \Cref{lem:admissible}~(2), and an edge $e = \{v_1, v_2\} \in C$.
Let $\{v_0, v_1, v_2\}$ be the face in the inside of $C$, or in $F_C$.
Since $\alpha$ is a $4$-coloring,
the color $\alpha(v_0)$ is different from both $\alpha(v_1)$ and $\alpha(v_2)$.
Therefore
$v_0$ does not belong to $C$ by \Cref{lem:admissible}~(1),
implying that $\St^2(v_0) \subseteq F_C$.
Since $C$ is an innermost closed trail, no edge incident to the vertex $v_0$ is nonsingular with respect to $\alpha$.
Thus, by \Cref{lemma:3-coloring}~(1), we can change the color of $v_0$.

Let $\alpha'$ be the $4$-coloring obtained from $\alpha$ by changing the color of $v_0$.
For each $v \in N(v_0)$,
we have $\# \left(\delta(v) \cap \Lk(v_0)\right) = 2$,
and denote $\delta(v) \cap \Lk(v_0)$ by $P_v$.
We define $\pi' = \bigcup_{v \in V} \pi_v'$ by
\begin{align*}
    \pi_v' :=
    \begin{cases}
        \pi_v & \text{if $v \notin N(v_0)$},\\
        \pi_v \cup \{ P_v \} & \text{if $v \in N(v_0)$ and $\NS_\alpha(v) \cap \Lk(v_0) = \emptyset$},\\
        (\pi_v \setminus \{P\}) \cup \{ P \bigtriangleup P_v \} & \text{if $v \in N(v_0)$ and $\pi_v$ contains $P$ with $|P \cap P_v| = 1$},\\
        \pi_v \setminus \{ P_v \} & \text{if $v \in N(v_0)$ and $\pi_v$ contains $P_v$}.
    \end{cases}
\end{align*}
See also Figure~\ref{fig:recolor}.
Then
$\pi'$ is an NS-pairing with respect to $\alpha'$ by \Cref{lem:nonsingular}.

Moreover, we can see that $\pi'$ is admissible as follows.
It is clear that
$\pi_v' \cap \pi_v$ satisfies (A1) and (A2) for each $v \in V$,
implying that $\pi_v'$ satisfies (A1) and (A2) if $v \notin N(v_0)$, or $v \in N(v_0)$ and $\pi_v$ contains $P_v$.
Since the edge $\{v_0, v\}$ is singular with respect to $\alpha'$,
the path $P_v$ (resp.\ $P \bigtriangleup P_v$)
does not cross any $P' \in \pi_v$ (resp.\ $P' \in \pi_v \setminus \{P\}$); $\pi_v'$ satisfies (A2) even for other $v$.
Suppose that $P_v = \{ \{u, v\}, \{v, w\} \}$.
Then, we have
$\varepsilon_{\alpha'}(\{u,v,v_0\}) = \varepsilon_{\alpha}(\{w,v,v_0\})$
and $\varepsilon_{\alpha'}(\{w,v,v_0\}) = \varepsilon_{\alpha}(\{u,v,v_0\})$.
This implies that,
if $v \in N(v_0)$ and $\pi_v$ contains $P$ with $|P \cap P_v| = 1$, 
then
$P \bigtriangleup P_v$ consists of one $+$-nonsingular edge and one $-$-nonsingular edge with respect to $\alpha'$,
and
if $v \in N(v_0)$ and $\NS_\alpha(v) \cap \Lk(v_0) = \emptyset$, 
then
$P_v$ consists of one $+$-nonsingular edge and one $-$-nonsingular edge with respect to $\alpha'$.
Thus $\pi_v'$ satisfies (A1) for other $v$.

Let $\C'$ be the set of the closed trails in $\C_{\pi'}$ containing some $e \in \Lk(v_0) \cap \NS_{\alpha'}$.
Then, we have $F_C = \bigcup_{C' \in \C'} F_{C'} \cup \St^2(v_0)$
and
$\C_{\pi'} = \C_\pi \setminus \{C\} \cup \C'$.
Therefore, we obtain $\vol(\C_{\pi'}) = \vol(\C_\pi) - \#\St^2(v_0) < \vol(\C_\pi)$;
see also Figure~\ref{fig:volume}
\begin{figure}
\centering
\includegraphics{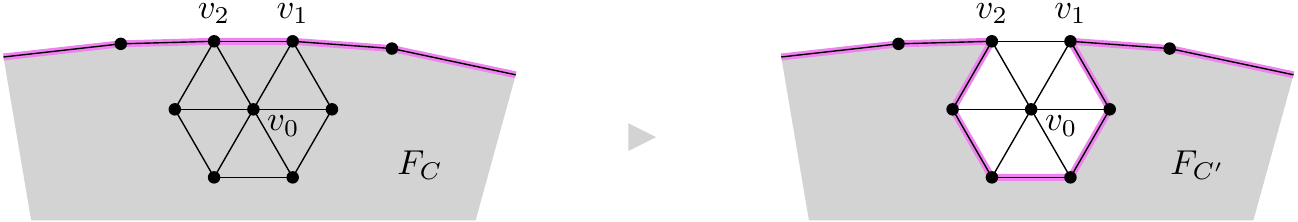}
\caption{Reducing the volume.}
\label{fig:volume}
\end{figure}

This completes the proof of the claim (and hence that of \Cref{thm:main}).
\end{proof}

Our proof of \Cref{thm:main} is constructive;
for a $4$-coloring $\alpha$ satisfying the balanced condition~(B),
we explicitly
construct
a sequence of single-changes from $\alpha$ to a certain $3$-coloring $\alpha^*$.
This leads to the following.
\begin{theorem}\label{thm:diameter}
Let $G$ be a $3$-colorable triangulation of the $2$-sphere.
For any $\alpha$ and $\beta$ belonging to the $3$-coloring component of $G$,
we can obtain in $O({\# V}^2)$ time a sequence of single-changes of length $O({\# V}^2)$ from $\alpha$ to $\beta$.
In particular,
the diameter of the $3$-coloring component of $G$ is $O({\# V}^2)$.
\end{theorem}
\begin{proof}
Let $\alpha$ and $\beta$ be $4$-colorings belonging to the $3$-coloring component of $G$.
It suffices to show that
there exists a sequence of single-changes from $\alpha$ to $\beta$ whose length is $O({\# V}^2)$,
and that it can be obtained in $O({\# V}^2)$ time.

For $\alpha$,
we can construct an admissible NS-pairing $\pi$
and the laminar family $\calL_\pi \subseteq 2^F$ in $O({\# V}^2)$ time.
By $\# \calL_\pi = O({\# F})$ (see e.g.,~\cite[Theorem 3.5]{Sch03}),
we have $\vol(\C_\pi) = \sum_{F_C \in \calL_\pi} {\# F_C} = O({\# F}^2)$.
Since $\vol(\C_\pi)$ strictly decreases for each single-change,
the length of a sequence of single-changes to obtain a $3$-coloring $\alpha^*$ from $\alpha$
is $O({\# F}^2)$,
which implies that it takes $O({\# F}^2)$ time to obtain $\alpha^*$ from $\alpha$.
Similarly,
we can obtain a $3$-coloring $\beta^*$ from $\beta$ in $O({\# V}^2)$ time by a sequence of single-changes, whose length is $O({\# V}^2)$.
Furthermore, the number of single-changes required to obtain the $3$-coloring $\beta^*$ from the $3$-coloring $\alpha^*$ is $O({\# V})$ (see Introduction for an actual sequence of single-changes to obtain $\beta^*$ from $\alpha^*$).
Hence, by concatenating the sequence from $\alpha$ to $\alpha^*$,
that from $\alpha^*$ to $\beta^*$, and the inverse of that from $\beta$ to $\beta^*$,
we obtain in $O({\# V}^2)$ time a sequence from $\alpha$ to $\beta$,
whose length is $O({\# V}^2)$.
\end{proof}

\Cref{thm:main,thm:diameter} immediately imply the polynomial-time solvability of \fourrecolor for $G$
if given $\alpha$ or $\beta$ belongs to the $3$-coloring component.
We here note that, for a $4$-coloring $\alpha$ of $G$,
we can check if it satisfies the balanced condition (B) in $O(\#F) = O(\#V)$ time.
\begin{corollary}
\label{cor:4Recoloring}
Let $G$ be a $3$-colorable triangulation of the $2$-sphere.
\fourrecolor for $G$ can be solved in $O(\#V)$ time,
provided one of the input $4$-colorings $\alpha$ and $\beta$ belongs to the $3$-coloring component of $\mathcal{R}_4(G)$.
In addition, if both $\alpha$ and $\beta$ belong to the $3$-coloring component,
then we can obtain a reconfiguration sequence from $\alpha$ to $\beta$ in $O({\# V}^2)$ time.
\end{corollary}

\subsection{High-dimensional case}
\label{subsec:highdim_case}
This subsection is devoted to extending the results in Section~\ref{subsec:2dim_case} to the case of triangulations of the $d$-sphere $S^d$ for $d=k-2\geq 3$.
In fact, we deal with not only $S^d$ but also connected closed $d$-manifolds $M$ with $H_{d-1}(M;\Z/2\Z)=\{0\}$;
see \Cref{sec:homology} for the definition of the homology groups with $\Z/2\Z$-coefficients.
By Poincar\'e duality and the universal coefficient theorem (\cite[Theorems~65.1 and 53.5]{Mun84}), $H_{d-1}(M;\Z/2\Z)=\{0\}$ is equivalent to $H_{1}(M;\Z/2\Z)=\{0\}$.
It is worth mentioning that there are infinitely many such manifolds other than $S^d$ when $d\geq 3$ (see e.g., \cite{Sav02}; \Cref{sec:homology} provides an example of such a $3$-manifold).
Note that the terminologies introduced in Section~\ref{sec:preliminaries} are valid for such manifolds.

We first review notations in high-dimensional cases. 
Let $[01\cdots d]$ denote the oriented $d$-simplex.
Then, its boundary is expressed as follows:
\[
\bigcup_{i=0}^d(-1)^i[01\cdots i-1\,i+1\cdots d],
\]
where the minus sign indicates the opposite orientation.
For instance, $-[0124]=[1024]$ as an  oriented $3$-simplex.
We now consider a triangulation $K$ of $M$ and let $\alpha\colon V \to \{0,1,\dots,d+1\}$ be a $(d+2)$-coloring of the $1$-skeleton of $K$.
For each $(d-1)$-simplex $\sigma^{d-1}=[v_0v_1\cdots v_{d-1}]$, define $\varepsilon_\alpha(\sigma^{d-1}):=+1$ if $[\alpha(v_0)\alpha(v_1)\cdots \alpha(v_{d-1})]$ appears in the above union, and $\varepsilon_\alpha(\sigma^{d-1}):=-1$ otherwise.
For instance, if a $3$-simplex $\sigma^{3}$ is expressed as $[1024]$ under $\alpha$, then $\varepsilon_\alpha(\sigma^{3})=-1$.

\begin{remark}
In this paper, a triangulation means a simplicial triangulation which is not necessarily combinatorial (or piecewise linear), that is, we do not require that the link $\Lk(\sigma^q)$ of each $q$-simplex $\sigma^q$ is PL-homeomorphic to the $(d-q-1)$-sphere.
On the other hand, the link $\Lk(\sigma^{d-2})$ is always PL-homeomorphic to the circle.
Note that there is no difference between simplicial and combinatorial triangulations in dimension less than or equal to four.
For more details we refer the reader to \cite[Section~1]{Man18}.
\end{remark}

The next theorem is an extension of Theorem~\ref{thm:main} to high dimensions.

\begin{theorem}
\label{thm:highdim_case}
Let $M$ be a connected closed $d$-manifold with $H_{d-1}(M;\Z/2\Z)=\{0\}$ admitting a triangulation $K$ whose $1$-skeleton $G$ is $(d+1)$-colorable.
Let $\alpha\colon V \to \{0,1,\dots,d+1\}$ be a $(d+2)$-coloring of $G$.
Then $\alpha$ belongs to the $(d+1)$-coloring component of $\R_{d+2}(G)$ if and only if
\begin{align}\label{eq:highdim B}
\#\{\sigma^d \in \St^d(\sigma^{d-2}) \mid \varepsilon_\alpha(\sigma^d) = +1 \}
= \#\{\sigma^d \in \St^d(\sigma^{d-2}) \mid \varepsilon_\alpha(\sigma^d) = -1 \}
\end{align}
for each $(d-2)$-simplex $\sigma^{d-2}$ in $K$.
\end{theorem}

\begin{remark}
Under the assumption of Theorem~\ref{thm:highdim_case}, the $(d+1)$-coloring component is nonempty (see \cite[Theorem~55(1) or Problem~2 in Section~VI.5]{Fis77}).
\end{remark}

\begin{example}
\label{ex:highdim}
For even integers $m,n\geq 3$, the join $C_m\ast C_n$ is $4$-colorable and can be realized by the $1$-skeleton of a triangulation of $S^3$.
Let $\alpha$ be a $5$-coloring.
Then, we may assume that $C_m$ and $C_n$ is colored with $\{0,1,2\}$ and $\{3,4\}$, respectively.
The coloring $\alpha$ induces a continuous map $C_m \to C_3$, and its ``degree'' is defined as discussed in \cite{Fis73I}.
Here, the balanced condition is equivalent to the degree being zero.
Thus, Theorem~\ref{thm:highdim_case} implies that a $5$-coloring $\alpha$ is single-equivalent to a $4$-coloring if and only if the degree is zero.
Note that this consequence can be directly confirmed without Theorem~\ref{thm:highdim_case}.

It is worth mentioning here that all $5$-colorings of $C_6\ast C_4$ are Kempe-equivalent.
Hence, this example shows a difference between the Kempe-equivalence and the reconfiguration discussed in this paper.
\end{example}

\begin{proof}[Proof of Theorem~\ref{thm:highdim_case}]
First note that a set $C$ of $(d-1)$-simplices can be identified with a $(d-1)$-chain with $\Z/2\Z$-coefficients and that if $C$ is a $(d-1)$-cycle then the hypothesis $H_{d-1}(M;\Z/2\Z)=\{0\}$ implies that the homology class $[C]$ is zero, that is, $M\setminus |C|$ admits a checkerboard coloring (see Appendix~\ref{sec:homology}).

If $\alpha$ lies in the $(d+1)$-coloring component, then the equality~\eqref{eq:highdim B} in the statement holds.
To show its converse, let $\alpha$ be a $(d+2)$-coloring satisfying the equality~\eqref{eq:highdim B}.
Fix a $d$-simplex $\sigma^d_{\mathrm{out}} \in K$.
For a $(d-1)$-cycle $C$ with $\Z/2\Z$-coefficients, we consider a checkerboard coloring with $\sigma^d_{\mathrm{out}}$ white and define $F_C$ to be the black $d$-simplices.
Also, for a set $\C$ of $(d-1)$-cycles, the volume $\vol(\C)$ is defined in the same way as in Section~\ref{subsec:2dim_case}.
Suppose that an admissible NS-pairing $\pi=\bigcup_{\sigma^{d-2}\in K}\pi_{\sigma^{d-2}}$ is given.
Then $\C_\pi$ and $\calL_\pi:=\{F_C \mid C \in \C_\pi\}$ are defined as well, and Lemma~\ref{lem:admissible} is generalized as follows: The restriction of $\alpha$ to $C$ is a $d$-coloring, and the family $\calL_\pi$ is laminar.
Now, in much the same way as Theorem~\ref{thm:main}, one can find a recolorable vertex $v \in K$ such that the $(d+2)$-coloring $\alpha'$ obtained from $\alpha$ by recoloring $v$ has an admissible NS-pairing $\pi'$ satisfying $\vol(\C_{\pi'})<\vol(\C_{\pi})$.

Let us show that there exists an admissible NS-pairing $\pi$.
Let $\sigma^{d-2} \in K$.
A $d$-simplex in $\St^d(\sigma^{d-2})$ is uniquely expressed as a join $\tau^1\ast\sigma^{d-2}$ for some $\tau^1 \in \Lk(\sigma^{d-2})$.
Hence, we have a projection from $\St(\sigma^{d-2})$ to the disk $\Lk(\sigma^{d-2})\ast\{v\}$ centered at $v$, which sends $\tau^1\ast\sigma^{d-2}$ to $\tau^1\ast\{v\}$.
Since the equality~\eqref{eq:highdim B} implies the balanced condition (B) around $v$ as illustrated in Figure~\ref{fig:projection}, we have an admissible NS-pairing $\pi$ of $(d-1)$-simplices around $\sigma^{d-2}$.
\end{proof}

\begin{figure}
\centering
\includegraphics{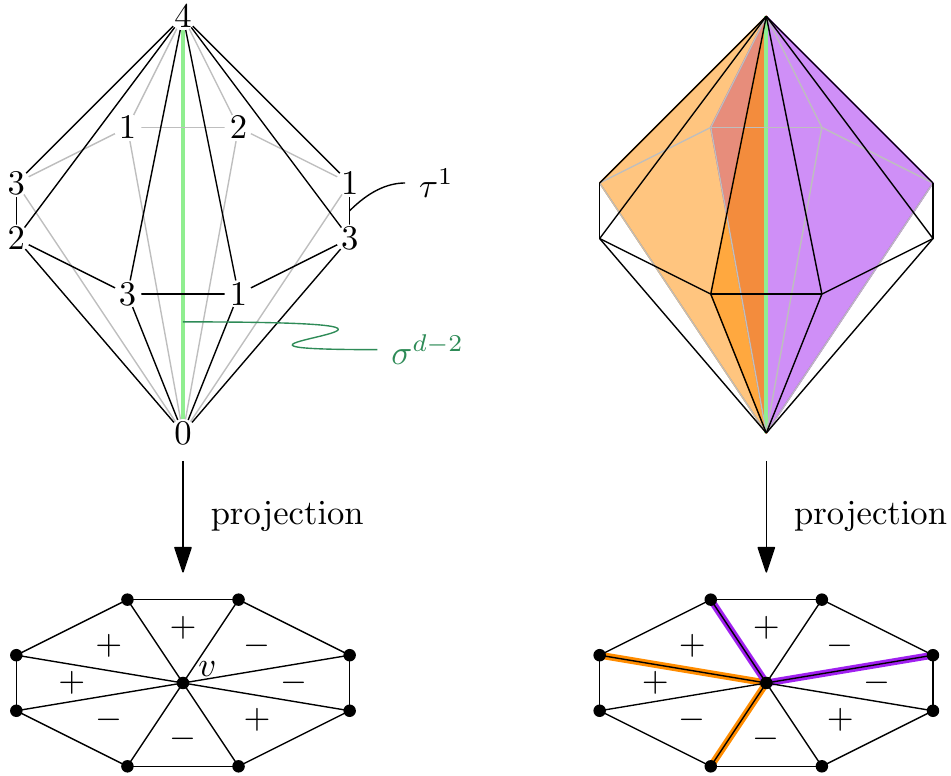}
\caption{An example of the projection when $d=3$. The figure on the right indicates an admissible NS-pairing.}
\label{fig:projection}
\end{figure}

We now have the following consequence.
Recall that $d=k-2$.

\begin{theorem}
\label{thm:diam_high}
Let $G$ be a $(k-1)$-colorable triangulation of the $(k-2)$-sphere for some $k \geq 4$.
For any $k$-colorings $\alpha, \beta$ of $G$ single-equivalent to some $(k-1)$-coloring, there is a sequence of single-changes of length $O(\#V^{2\lfloor (k-1)/2 \rfloor})$ which transforms $\alpha$ to $\beta$.
\end{theorem}

\begin{proof}
It follows from the upper bound theorem \cite{stanley-ubc} that the number of $(k-2)$-simplices in the triangulation is $O(\#V^{\lfloor (k-1)/2 \rfloor})$.
Then, we obtain the upper bound $O(\#V^{2\lfloor (k-1)/2 \rfloor})$ in much the same argument as Theorem~\ref{thm:diameter}.
\end{proof}

As in \Cref{subsec:2dim_case},
\Cref{thm:highdim_case,thm:diam_high} immediately imply the following:
\begin{corollary}
\label{cor:kRecoloring}
Let $G$ be a $(k-1)$-colorable triangulation of the $(k-2)$-sphere for some $k \geq 4$.
\krecolor for $G$ can be solved
in $O(\# V^{\lfloor (k-1)/2 \rfloor})$ time,
provided one of the input $k$-colorings $\alpha$ and $\beta$ belongs to the $(k-1)$-coloring component of $\mathcal{R}_k(G)$.
In addition, if both $\alpha$ and $\beta$ belong to the $(k-1)$-coloring component,
then we can obtain a reconfiguration sequence from $\alpha$ to $\beta$
$O(\# V^{2\lfloor (k-1)/2 \rfloor})$ time.
\end{corollary}

Note that the proofs of \Cref{thm:highdim_case,thm:diam_high} say that,
if we consider the input size as the size $\# K$ of the simplical complex $K$,
then we can solve \krecolor in linear time in the input size in our setting.

\section{Connectedness of the $4$-coloring reconfiguration graph}
\label{sec:connR}

In this section, we solve the second question posed in Introduction:
\emph{In what $3$-colorable triangulation of the $2$-sphere all
$4$-colorings are single-equivalent?}
To explain the answer, we introduce some notations.
Since we deal with only the case of the $2$-sphere in this section,
we simply use the term a \emph{triangulation}
instead of a triangulation of the $2$-sphere.

A \emph{separating triangle} in a triangulation is
a cycle of length $3$ that does not bound a face.
Note that a triangulation with at least five vertices is $4$-connected if and only if 
it has no separating triangles.
A triangulation with a separating triangle $C$ can be split into two triangulations,
the subgraph induced by the inside of $C$ and that by the outside of $C$,
respectively.
Note that they share $C$.
By iteratively applying this procedure to a triangulation $G$ with $k$ separating triangles,
we obtain a collection of $k+1$ triangulations without separating triangles.
We call the $k+1$ triangulations \emph{$4$-connected pieces} of $G$.
It is known~\cite{Cunningham} that the collection of the $4$-connected pieces are uniquely determined.
It is easy to see that 
$G$ is $3$-colorable
if and only if 
every $4$-connected piece of $G$ is $3$-colorable.

The \emph{octahedral graph} is the $1$-skeleton of the octahedron (\figurename~\ref{fig:octahedral_example1}),
which has six vertices, twelve edges, and eight faces,
and is $3$-colorable.
The following is the main theorem in this section.

\begin{figure}
    \centering
    \includegraphics{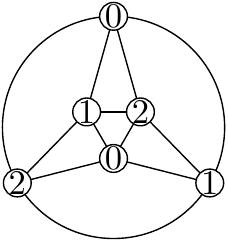}
    \caption{The octahedral graph with a $3$-coloring.}
    \label{fig:octahedral_example1}
\end{figure}

\begin{theorem}
\label{thm:conn}
Let $G$ be a $3$-colorable triangulation.
Then, $\mathcal{R}_4(G)$ is connected
if and only if 
every $4$-connected piece of $G$ is isomorphic to the octahedral graph.
\end{theorem}

Let $G$ be a $3$-colorable triangulation.
We say a $4$-coloring $\alpha$ not belonging to
the $3$-coloring component $\mathcal{R}_4(G)$ to be \emph{unbalanced}
since it follows from Theorem \ref{thm:main} that
$\alpha$
violates
the balanced condition (B).
It is easy to see that 
$\mathcal{R}_4(G)$ is connected if and only if $G$ has no unbalanced $4$-coloring.
Thus, Theorem \ref{thm:conn} is equivalent to the following statement.
\begin{theorem}
\label{thm:conn2}
Let $G$ be a $3$-colorable triangulation.
Then, $G$ has no unbalanced $4$-coloring
if and only if 
every $4$-connected piece of $G$ is isomorphic to the octahedral graph.
\end{theorem}

We will prove \Cref{thm:conn2},
by combining some lemmas
together with the so-called generating theorem.
The following is a well-known property on the number of $+$- and $-$-faces:
\begin{lemma}[{\cite[Lemma~1]{Fis73II}}]
\label{lem:FiskII}
Let $G$ be a triangulation of the $2$-sphere
and $\alpha$ a $4$-coloring of $G$.
Then, for every vertex $v$ of $G$, we have
\[
\#\varepsilon_{\alpha, v}^{-1}(+1) \equiv \#\varepsilon_{\alpha, v}^{-1}(-1)
\pmod{3}.
\]
\end{lemma}

The following deals with splitting a triangulation to obtain a $4$-connected piece.
\begin{lemma}
\label{lemma:conn2}
Let $G$ be a $3$-colorable triangulation with a separating triangle $C$,
and let $G_1$ and $G_2$ be the two triangulations obtained by splitting along $C$.
Then $G$ has an unbalanced $4$-coloring
if and only if 
$G_1$ or $G_2$ has an unbalanced $4$-coloring.
\end{lemma}

\begin{proof}
For each $i =1,2$, let $f_i$ be the face of $G_i$ bounded by $C$, see \figurename~\ref{fig:separating_triangle1}.
Note that for each vertex $x$ of $G$,
we have 
$\St_{G}^2(x) = \St_{G_i}^2(x)$ if $x \in V(G_i) \setminus V(C)$ for some $i =1,2$,
and 
$\St_G^2(x) = (\St_{G_1}^2(x)\setminus \{f_1\}) \cup (\St_{G_2}^2(x)\setminus \{f_2\})$
otherwise.

\begin{figure}
    \centering
    \includegraphics{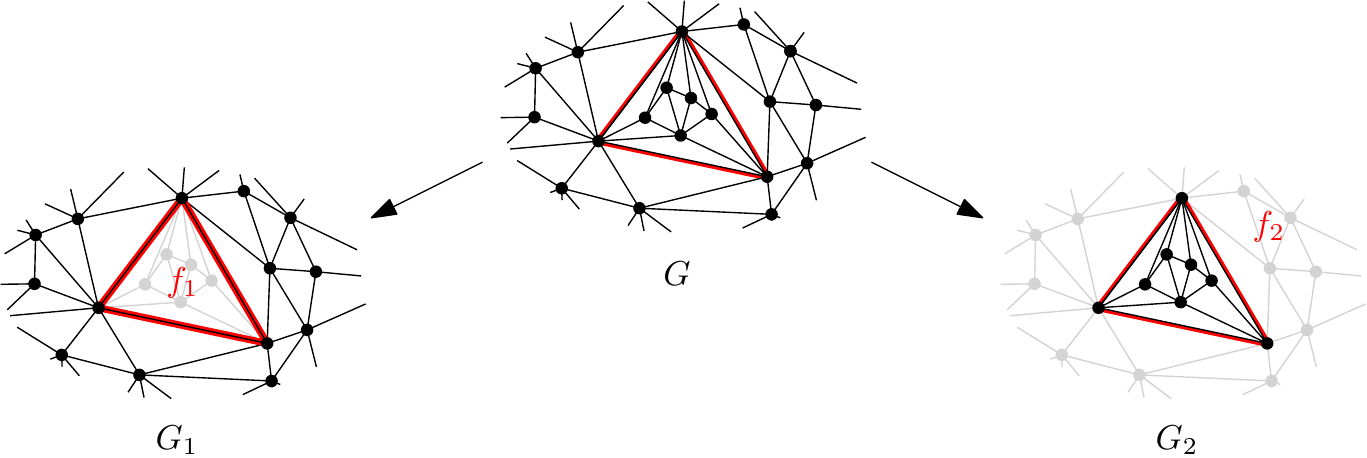}
    \caption{Proof of \Cref{lemma:conn2}. The separating triangle $C$ is highlighted by red lines.}
    \label{fig:separating_triangle1}
\end{figure}

We first prove the if part.
By symmetry, suppose that $G_1$ has an unbalanced $4$-coloring $\alpha_1$.
Since $G$ is $3$-colorable, $G_2$ admits a $3$-coloring, say $\alpha_2$.
Since the three vertices of $C$ receive three distinct colors 
both in $\alpha_1$ and in $\alpha_2$,
we may assume that the colors of the vertices of $C$ coincide in $\alpha_1$ and in $\alpha_2$.
We construct a $4$-coloring $\alpha$ of $G$ as follows:
$\alpha(x) = \alpha_1(x)$ if $x \in V(G_1)$, and 
$\alpha(x) = \alpha_2(x)$ otherwise.
Since $\alpha_1$ is an unbalanced $4$-coloring of $G_1$,
it follows from Theorem \ref{thm:main} that there exists a vertex $x$ of $G_1$ 
that
violates
the equality~\eqref{eq:iff-condition} for $\alpha_1$.
If $x$ is contained in $V(G_1) \setminus V(C)$,
then 
$x$
violates
the equality~\eqref{eq:iff-condition} even for $\alpha$,
and hence $\alpha$ is an unbalanced $4$-coloring of $G$ by Theorem \ref{thm:main}.
Suppose $x \in V(C)$.
By the definition of $\alpha$, 
we have
\begin{eqnarray*}
\lefteqn{
\#\varepsilon_{\alpha, x}^{-1}(+1) - \#\varepsilon_{\alpha, x}^{-1}(-1)
}
\\
&=&
\#\varepsilon_{\alpha_1, x}^{-1}(+1) - \#\varepsilon_{\alpha_1, x}^{-1}(-1)
- \varepsilon_{\alpha_1}(f_1)
+ \#\varepsilon_{\alpha_2, x}^{-1}(+1) - \#\varepsilon_{\alpha_2, x}^{-1}(-1)
- \varepsilon_{\alpha_2}(f_2).
\end{eqnarray*}
Since $\alpha_2$ is a $3$-coloring of $G_2$,
we obtain
$
\#\varepsilon_{\alpha_2, x}^{-1}(+1)
=
\#\varepsilon_{\alpha_2, x}^{-1}(-1)
$.
Thus, if $x$ satisfies the equality~\eqref{eq:iff-condition} even for $\alpha$,
then  
$\#\varepsilon_{\alpha, x}^{-1}(+1)
=
\#\varepsilon_{\alpha, x}^{-1}(-1)
$,
and therefore,
\begin{eqnarray*}
\#\varepsilon_{\alpha_1, x}^{-1}(+1)
- 
\#\varepsilon_{\alpha_1, x}^{-1}(-1)
=
\varepsilon_{\alpha_1}(f_1)
+ \varepsilon_{\alpha_2}(f_2)
\in \{-2, 0, 2\}.
\end{eqnarray*}
By \Cref{lem:FiskII},
we have
$
\#\varepsilon_{\alpha_1, x}^{-1}(+1)
- 
\#\varepsilon_{\alpha_1, x}^{-1}(-1)
\equiv 0 \pmod{3}$,
and hence $x$ satisfies the equality~\eqref{eq:iff-condition} for $\alpha_1$,
a contradiction.
Therefore, 
$x$
violates
the equality~\eqref{eq:iff-condition} even for $\alpha$,
and it follows from Theorem \ref{thm:main} that 
$\alpha$ is an unbalanced $4$-coloring of $G$.
This proves the if part.

We next prove the only-if part by contrapositive.
Suppose that $G_i$ has no unbalanced $4$-coloring for any $i = 1,2$.
Let $\alpha$ be a $4$-coloring of $G$,
and we show that $\alpha$ is not unbalanced.
By Theorem \ref{thm:main},
it suffices to show that 
each vertex $v$ of $G$ satisfies the equality~\eqref{eq:iff-condition}.
For each $i  =1,2$,
let $\alpha_i$ be the restriction of $\alpha$ to the vertices of $G_i$.
Since $\alpha_i$ is not an unbalanced $4$-coloring of $G_i$,
each vertex in $V(G_i) \setminus V(C)$
satisfies the equality~\eqref{eq:iff-condition} for $\alpha_i$.
Let $x$ be a vertex contained in $C$.
For any $i  =1,2$,
since $\alpha_i$ is not an unbalanced $4$-coloring of $G_i$,
it follows from the equality~\eqref{eq:iff-condition} for $\alpha_i$ that
\[
\#\varepsilon_{\alpha_i, x}^{-1}(+1)
- 
\#\varepsilon_{\alpha_i, x}^{-1}(-1)
= 0.
\]
Therefore, we have
\begin{eqnarray*}
\lefteqn{
\#\varepsilon_{\alpha, x}^{-1}(+1) - \#\varepsilon_{\alpha, x}^{-1}(-1)
}\\
&=&
\#\varepsilon_{\alpha_1, x}^{-1}(+1)
- 
\#\varepsilon_{\alpha_1, x}^{-1}(-1)
-\varepsilon_{\alpha_1}(f_1)
+ 
\#\varepsilon_{\alpha_2, x}^{-1}(+1)
-
\#\varepsilon_{\alpha_2, x}^{-1}(-1)
- \varepsilon_{\alpha_2}(f_2)
\\
&\in& \{-2, 0, 2\}.
\end{eqnarray*}
By \Cref{lem:FiskII},
we obtain
$
\#\varepsilon_{\alpha, x}^{-1}(+1) - 
\#\varepsilon_{\alpha, x}^{-1}(-1)
\equiv 0 \pmod{3}$,
and hence $x$ satisfies the equality~\eqref{eq:iff-condition} for $\alpha$.
This proves that
the $4$-coloring $\alpha$ of $G$ is not  unbalanced,
and hence this completes the proof of the only-if part.
\end{proof}

By Lemma \ref{lemma:conn2},
we can focus on a $4$-connected $3$-colorable triangulation $G$.

We now define two operations to reduce $G$
to a smaller triangulation $G'$ as follows.
Here a cycle consisting of $\ell$ vertices $v_1,v_2, \dots , v_{\ell}$
and ${\ell}$ edges $\{v_1, v_{2}\}, \{v_2, v_{3}\}, \dots , \{v_{\ell-1}, v_{\ell}\}, \{v_{\ell}, v_1\}$
is denoted by the sequence $v_1v_2 \cdots v_{\ell}v_1$.
Let $v$ be a vertex of degree four in $G$ and let $w_1w_2w_3w_4w_1$ be the cycle that forms the link of $v$.
The \emph{$4$-contraction} of $v$ at $\{w_1, w_3\}$,
illustrated in \figurename~\ref{fig:four_contraction},
is to remove $v$,
identify the vertices $w_1$ and $w_3$, and
replace the two pairs of multiple edges obtained from $\{ \{w_1,w_2\}, \{w_2,w_3\}\}$ and $\{\{w_1,w_4\}, \{w_3,w_4\}\}$
with two single edges, respectively.
\begin{figure}
\centering
\includegraphics{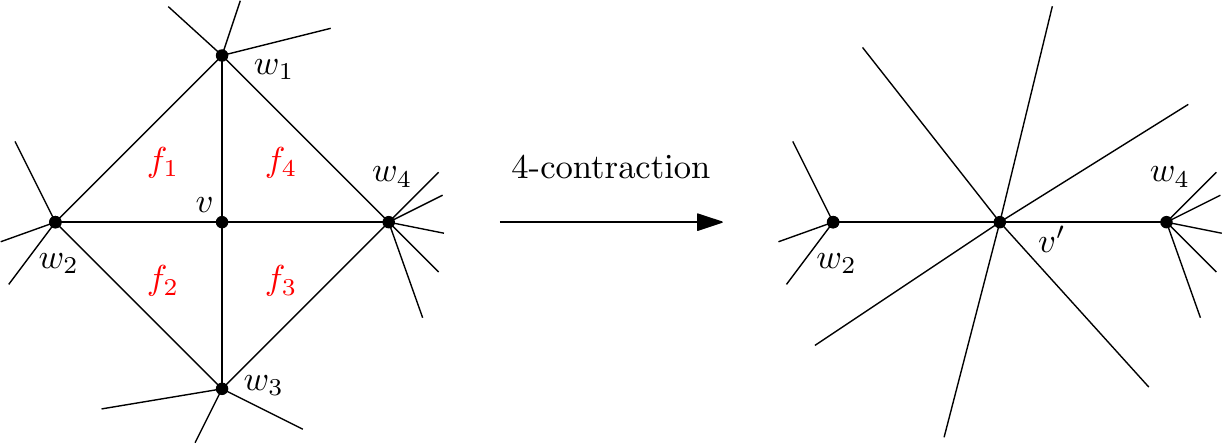}
\caption{The $4$-contraction of $v$ at $\{w_1, w_3\}$.}
\label{fig:four_contraction}
\end{figure}
Let $u$ and $v$ be adjacent vertices of degree four,
where $w_1w_2w_3vw_1$ and $w_1uw_3w_4w_1$ are the cycles
that form the links of $u$ and $v$, respectively.
The \emph{twin-contraction} of $\{u,v\}$ at $\{w_1, w_3\}$,
illustrated in \figurename~\ref{fig:twin_contraction},
is to remove $u$ and $v$,
identify the vertices $w_1$ and $w_3$, and
replace the two pairs of multiple edges obtained from $\{\{w_1,w_2\}, \{w_2,w_3\}\}$ and $\{\{w_1,w_4\}, \{w_3,w_4\}\}$
with two single edges, respectively.
\begin{figure}
\centering
\includegraphics{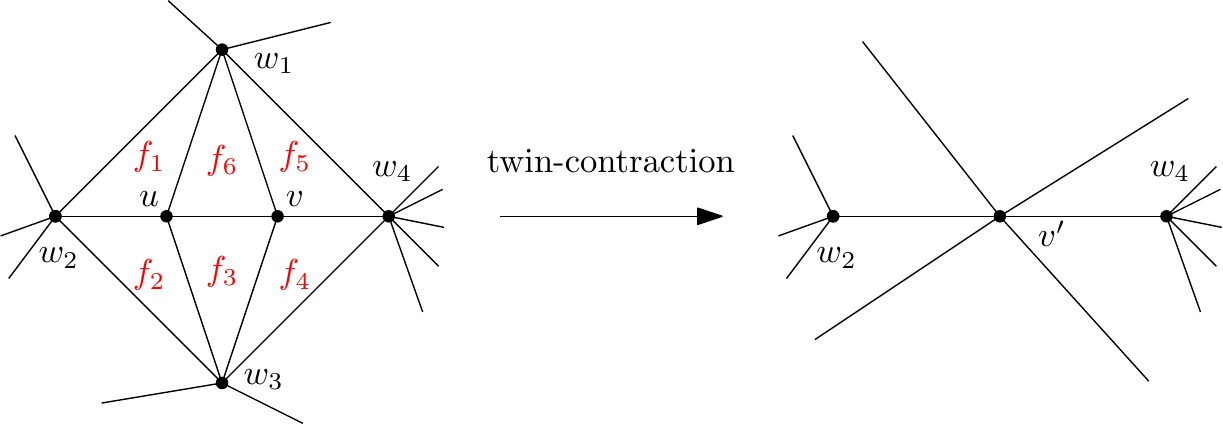}
\caption{The twin-contraction of $\{u, v\}$ at $\{w_1, w_3\}$.}
\label{fig:twin_contraction}
\end{figure}

Notice that we do not perform these operations 
if they give rise to multiple edges.
Matsumoto and Nakamoto proved the following generating theorem.

\begin{theorem}[\cite{Matsumoto}]
\label{thm:Matsu}
For every $4$-connected $3$-colorable triangulation $G$,
there exists a sequence $G_0, G_1, \dots , G_{\ell}$
from $G_{0} := G$
such that 
$G_{\ell}$ is the octahedral graph,
$G_i$ is a $4$-connected $3$-colorable triangulation for $0 \leq i \leq \ell$, 
and
$G_i$ is obtained from $G_{i-1}$ by either a $4$-contraction or a twin-contraction
for $1 \leq i \leq \ell$.
\end{theorem}

Since the next lemma deals with only triangulations of order at most nine,
it can be easily shown.
For convenience, its proof is given in \Cref{sec:ap:proof}.
Recall that 
the \emph{double wheel} is a triangulation obtained from a cycle on the plane
by adding a new vertex to each of the inside and the outside of the cycle 
and connecting the added vertices to the all vertices of the cycle.
Note that the octahedral graph is the double wheel of order $6$.

\begin{lemma}
\label{lemma:small_tri}
Each of the following holds.
\begin{itemize}
\item[{\rm (1)}]
The octahedral graph admits no unbalanced $4$-coloring.
\item[{\rm (2)}]
If the octahedral graph is obtained
from a $3$-colorable triangulation $G$ by a $4$-contraction,
then 
$G$
is isomorphic to the double wheel of order $8$.
\item[{\rm (3)}]
The double wheel of order $8$ admits an unbalanced $4$-coloring.
\item[{\rm (4)}]
If the octahedral graph is obtained 
from a $3$-colorable triangulation $G$ by a twin-contraction,
then $G$ has a separating triangle.
\end{itemize}
\end{lemma}

We prove the next lemma
for a $4$-contraction and a twin-contraction.
\begin{lemma}
\label{lemma:splitting}
Let $G$ be a $4$-connected $3$-colorable triangulation,
and let $G'$ be a $4$-connected $3$-colorable triangulation obtained from $G$
by either a $4$-contraction or a twin-contraction.
If $G'$ has an unbalanced $4$-coloring, then so does $G$.
\end{lemma}

\begin{proof}
Suppose first that $G'$ is obtained from $G$
by the $4$-contraction of a vertex $v$ at $\{w_1, w_3\}$,
where $w_1w_2w_3w_4w_1$ is the cycle that forms the link of $v$.
Let $v'$ be the vertex obtained by the identification of $w_1$ and $w_3$,
and let $f_1$, $f_2$, $f_3$, and $f_4$ be the faces of $G$
that are bounded by the cycles $w_1w_2vw_1, w_2w_3vw_2, w_3w_4vw_3$ and $w_4w_1vw_4$, respectively. See Figure~\ref{fig:four_contraction}.

By the assumption, $G'$ has an unbalanced $4$-coloring $\alpha'$.
We define $\alpha : V(G) \to \{0,1,2,3\}$ by
\begin{align*}
    \alpha(x) :=
    \begin{cases}
    \alpha'(x) & \text{if $x \in V(G) \setminus \{v,w_1,w_3\}$},\\
    \alpha'(v') & \text{if $x \in \{w_1,w_3\}$},\\
    c & \text{if $x = v$},
    \end{cases}
\end{align*}
where $c$ is some color in $\{0,1,2,3\} \setminus \{\alpha'(v'), \alpha'(w_2), \alpha'(w_4) \}$.
It is easy to see that $\alpha$ is a $4$-coloring of $G$.
We will show that 
$G$ has a vertex
that
violates
the equality~\eqref{eq:iff-condition} for $\alpha$.
Since $\alpha'$ is an unbalanced $4$-coloring of $G'$,
it follows from Theorem \ref{thm:main} that
there exists a vertex $x$ in $G'$ 
that
violates
the equality~\eqref{eq:iff-condition} for $\alpha'$.
If $x$ is contained in $V(G') \setminus \{v', w_2, w_4\}$,
then 
$\St_{G'}^2(x) = \St_{G}^2(x)$,
and hence 
$x$
violates
the equality~\eqref{eq:iff-condition} even for $\alpha$.
Next, consider the case $x \in \{w_2, w_4\}$.
By symmetry, suppose $x = w_2$ without loss of generality.
Note that
$
\left|
\#\varepsilon_{\alpha', w_2}^{-1}(+1)
-
\#\varepsilon_{\alpha', w_2}^{-1}(-1)
\right|
\geq 3$ by \Cref{lem:FiskII}
and that
$\St_{G'}^2(w_2) = \St_{G}^2(w_2)\setminus \{f_1, f_2\}$.
Thus,
\begin{eqnarray*}
\left|
\#\varepsilon_{\alpha, w_2}^{-1}(+1) -
\#\varepsilon_{\alpha, w_2}^{-1}(-1)
\right|
=
\left|
\#\varepsilon_{\alpha', w_2}^{-1}(+1) -
\#\varepsilon_{\alpha', w_2}^{-1}(+1) +
\varepsilon_{\alpha}(f_1) + \varepsilon_{\alpha}(f_2)
\right|
\geq
1.
\end{eqnarray*}
Therefore,
$x$
violates
the equality~\eqref{eq:iff-condition} even for $\alpha$.
For the case $x = v'$,
note that
$\St_{G'}^2(v') = (\St_{G}^2(w_1)\setminus \{f_1, f_4\}) \cup (\St_{G}^2(w_3)\setminus \{f_2, f_3\})$.
By the definition of $\alpha$, we have
$\varepsilon_\alpha(f_1) = - \varepsilon_\alpha(f_2)$
and $\varepsilon_\alpha(f_3) = - \varepsilon_\alpha(f_4)$.
If 
$
\#\varepsilon_{\alpha, w_1}^{-1}(+1)
=
\#\varepsilon_{\alpha, w_1}^{-1}(-1)$
and 
$
\#\varepsilon_{\alpha, w_3}^{-1}(+1)
=
\#\varepsilon_{\alpha, w_3}^{-1}(+1)$,
then
\begin{eqnarray*}
\lefteqn{
\#\varepsilon_{\alpha', v'}^{-1}(+1) -
\#\varepsilon_{\alpha', v'}^{-1}(-1)
}\\
&=&
\#\varepsilon_{\alpha, w_1}^{-1}(+1) + 
\#\varepsilon_{\alpha, w_3}^{-1}(+1) -
\#\varepsilon_{\alpha, w_1}^{-1}(-1) -
\#\varepsilon_{\alpha, w_3}^{-1}(-1)
-\varepsilon_{\alpha}(f_1) - \varepsilon_{\alpha}(f_2)
-\varepsilon_{\alpha}(f_3) - \varepsilon_{\alpha}(f_4)
\\
&=&
0,
\end{eqnarray*}
which contradicts the assumption that
$x=v'$ does not satisfy the equality~\eqref{eq:iff-condition} for $\alpha'$.
Therefore,
at least one of  $w_1$ and $w_3$ 
violates
the equality~\eqref{eq:iff-condition} for $\alpha$,
and $\alpha$ is an unbalanced $4$-coloring of $G$.

Suppose next that $G'$ is obtained from $G$
by the twin-contraction of $\{u,v\}$ at $\{w_1, w_3\}$,
where $w_1w_2w_3vw_1$ and $w_1uw_3w_4w_1$
are the cycles that form the links of $u$ and $v$, respectively.
Let $v'$ be the vertex obtained by the identification of $w_1$ and $w_3$,
and let $f_1, f_2, f_3, f_4, f_5, f_6$ be the faces of $G$
that are bounded by the cycles
$w_1w_2uw_1, w_2w_3uw_2, w_3vuw_3, w_3w_4vw_3, w_4w_1vw_4, w_1uvw_1$, respectively. See Figure~\ref{fig:twin_contraction}.
By the assumption, $G'$ has an unbalanced $4$-coloring $\alpha'$.
We define $\alpha : V(G) \to \{0,1,2,3\}$ by
\begin{align*}
    \alpha(x) :=
    \begin{cases}
    \alpha'(x) & \text{if $x \in V(G) \setminus \{v,w_1,w_3\}$},\\
    \alpha'(v') & \text{if $x \in \{w_1,w_3\}$},\\
    c & \text{if $x = u$},\\
    c' & \text{if $x = v$},
    \end{cases}
\end{align*}
where $c$ is some color in $\{0,1,2,3\} \setminus \{\alpha'(v'),\alpha'(w_2)\}$
and $c'$ is some color in $\{0,1,2,3\} \setminus \{\alpha'(v'), \alpha'(w_4), c\}$.
It is easy to see that $\alpha$ is a $4$-coloring of $G$.
Since $\alpha'$ is an unbalanced $4$-coloring of $G'$,
it follows from Theorem \ref{thm:main} that
there exists a vertex $x$ in $G'$ 
that
violates
the equality~\eqref{eq:iff-condition} for $\alpha'$.
If $x \neq v'$,
then the same argument as above implies that
$x$
violates
the equality~\eqref{eq:iff-condition} even for $\alpha$.
On the other hand,
suppose that $x = v'$.
Note that
$\St_{G'}^2(v') = (\St_{G}^2(w_1)\setminus \{f_1, f_5, f_6\}) \cup (\St_{G}^2(w_3)\setminus \{f_2, f_3, f_4\})$.
By the definition of $\alpha$, we have
$\varepsilon_\alpha(f_1) = - \varepsilon_\alpha(f_2)$,
$\varepsilon_\alpha(f_3) = - \varepsilon_\alpha(f_6)$,
and $\varepsilon_\alpha(f_4) = - \varepsilon_\alpha(f_5)$.
Thus,
if 
$
\#\varepsilon_{\alpha, w_1}^{-1}(+1)=
\#\varepsilon_{\alpha, w_1}^{-1}(-1)$
and 
$
\#\varepsilon_{\alpha, w_3}^{-1}(+1) =
\#\varepsilon_{\alpha, w_3}^{-1}(-1)$,
then
\begin{eqnarray*}
\#\varepsilon_{\alpha', v'}^{-1}(+1) - 
\#\varepsilon_{\alpha', v'}^{-1}(-1)
&=&
\#\varepsilon_{\alpha, w_1}^{-1}(+1) +
\#\varepsilon_{\alpha, w_3}^{-1}(-1) - 
\#\varepsilon_{\alpha, w_1}^{-1}(+1) +
\#\varepsilon_{\alpha, w_3}^{-1}(-1)
\\
&&
-\varepsilon_{\alpha}(f_1) - \varepsilon_{\alpha}(f_2)
-\varepsilon_{\alpha}(f_3) - \varepsilon_{\alpha}(f_4)
-\varepsilon_{\alpha}(f_5) - \varepsilon_{\alpha}(f_6)
\\
&=&
0,
\end{eqnarray*}
which contradicts the assumption that
$x=v'$
violates
the equality~\eqref{eq:iff-condition} for $\alpha'$.
This implies that 
at least one of $w_1$ and $w_3$ 
violates
the equality~\eqref{eq:iff-condition} for $\alpha$.
In both cases,
$\alpha$ is an unbalanced $4$-coloring of $G$,
and this completes the proof of Lemma \ref{lemma:splitting}.
\end{proof}

We are ready to prove Theorem \ref{thm:conn2},
which shows Theorem \ref{thm:conn} as mentioned before.
\begin{proof}[Proof of Theorem~\ref{thm:conn2}]
We first prove the if part.
Suppose that 
every $4$-connected piece of $G$ is isomorphic to the octahedral graph.
By Lemma \ref{lemma:small_tri}~(1),
every $4$-connected piece of $G$ admits no unbalanced $4$-coloring.
By recursively applying Lemma \ref{lemma:conn2},
we see that $G$ has no unbalanced $4$-coloring.
This completes the proof of the if part.

We next prove the only-if part.
Suppose that 
there is a $4$-connected piece $H$ of $G$ that is not isomorphic to the octahedral graph.
By Theorem \ref{thm:Matsu},
there exists a sequence $H_0, H_1, \dots , H_{\ell}$
from $H_0 := H$
such that
$H_{\ell}$ is the octahedral graph,
$H_i$ is a $4$-connected $3$-colorable triangulation for $0 \leq i \leq \ell$,
and
$H_i$ is obtained from $H_{i-1}$ by either a $4$-contraction or a twin-contraction
for $1 \leq i \leq \ell$.
Since $H$ is not isomorphic to the octahedral graph,
we have $\ell \geq 1$.
In particular, $H_{\ell -1}$ exists.
Since $H_{\ell -1}$ is $4$-connected and $3$-colorable,
it follows from Lemma \ref{lemma:small_tri}~(4) that
$H_{\ell -1}$ is obtained from the octahedral graph $H_{\ell}$ by a $4$-contraction.
By Lemma \ref{lemma:small_tri}~(2) and~(3),
$H_{\ell -1}$ is isomorphic to the double wheel of order $8$,
and admits an unbalanced $4$-coloring.
By recursively applying Lemma \ref{lemma:splitting}
to the $4$-connected $3$-colorable triangulations $H_{\ell -2}, \dots , H_{0}$,
the $4$-connected piece $H_{0} = H$ admits an unbalanced $4$-coloring.
Thus,
by applying Lemma \ref{lemma:conn2},
the triangulation $G$ also has an unbalanced $4$-coloring.
This completes the proof of the only if part.
\end{proof}

The criterion in Theorem \ref{thm:conn} can be used to obtain a polynomial-time, particularly, linear-time algorithm 
for \fourconn for a $3$-colorable triangulation of the $2$-sphere,
as follows.

\begin{corollary}
\label{thm:conn_algorithm}
\fourconn for a $3$-colorable triangulation of the $2$-sphere is solvable in $O(\# V)$ time.
\end{corollary}

\begin{proof}
By Theorem \ref{thm:conn},
it suffices to check
whether there exists a $4$-connected piece of $G$
that is not isomorphic to 
the octahedral graph.
This can be obtained
by enumerating all separating triangles
in linear time \cite{DBLP:journals/siamcomp/ChibaN85}.
\end{proof}

In addition to Corollary~\ref{thm:conn_algorithm},
the proof of Theorem~\ref{thm:conn} implies that 
if the answer to \fourconn is NO,
then we can find in polynomial time an unbalanced $4$-coloring $\alpha$ 
in a given $3$-colorable triangulation $G$.
Since $\alpha$ cannot belong to the $3$-coloring component of $\mathcal{R}_4(G)$,
this would be a certificate for being a NO-instance.
We leave the detail to the readers.

\section{PSPACE-completeness}\label{sec:pspace}

As in Introduction, we show the following result in this section.

\begin{theorem}\label{thm:pspace}
For $k\geq 4$, the problem \kplusrecolor for $(k-1)$-colorable triangulation of the $(k-2)$-sphere is PSPACE-complete.
\end{theorem}

When restricted to the case $k=4$, \Cref{thm:pspace} implies that \textsc{$5$-Recoloring} is PSPACE-complete even for planar $3$-colorable triangulations (i.e., even triangulations).

A path consisting of $\ell$ vertices $v_1,v_2, \dots , v_{\ell}$
and ${\ell}-1$ edges $\{v_1, v_{2}\}, \{v_2, v_{3}\}, \dots , \{v_{\ell-1}, v_{\ell}\}$
is denoted by the sequence $v_1v_2 \cdots v_{\ell}$ of vertices.

\subsection{\Lrecolor}\label{pre}

In order to prove Theorem~\ref{thm:pspace},
we introduce a new recoloring problem.
For a list coloring, we associate a \emph{list assignment} $L = (L(v))_{v \in V(G)}$ with a graph $G$ such that each $v\in V(G)$ is assigned a list $L(v)$ of colors.
For a list assignment $L$ of a graph $G$,
a map $\alpha$ on $V(G)$ is an \emph{$L$-coloring} if $\alpha(v) \in L(v)$ for every $v \in V(G)$ and $\alpha(u) \neq \alpha(v)$ for every $\{u,v\} \in E(G)$.
For a graph $G$ and a list assignment $L$ of $G$, the $L$-\emph{coloring reconfiguration graph}, denoted by $\mathcal{R}(G,L)$,
is defined as follows:
Its vertex set consists of all $L$-colorings of $G$
and there is an edge between two $L$-colorings $\alpha$ and $\beta$ of $G$
if and only if $\beta$ is obtained from $\alpha$ by recoloring only a single vertex in $G$.
We consider the following reconfiguration problem named \Lrecolor.
\begin{description}
        \item[\underline{\Lrecolor}]
\item[Input:]
A graph $G$, a list assignment $L$ of $G$, 
and two $L$-colorings $\alpha$ and $\beta$ of $G$. 
\item[Output:]
YES if $\alpha$ and $\beta$ are connected in $\mathcal{R}(G,L)$, and NO otherwise.
\end{description}

Let $P$ be a $(u,v)$-path with a list assignment $L$.
An $L$-coloring $\alpha$ of $P$ is a $(c,d)$\emph{-coloring} if $\alpha(u)=c$ and $\alpha(v)=d$.
For $a \in L(u)$ and $b \in L(v)$, we call a pair $(P,L)$ an $(a,b)$\emph{-forbidding path} if the following conditions are satisfied.

\begin{itemize}
\item[(I)] A $(c,d)$-coloring exists if and only if $(c,d) \neq (a,b)$,
\item[(II)] if both a $(c,d)$-coloring and a $(c^{\prime},d)$-coloring exist, then for any $(c,d)$-coloring,
a sequence of recolorings exists that ends with some $(c^{\prime},d)$-coloring, without ever recoloring $v$, and only recoloring $u$ at the last step, and
\item[(III)] if both a $(c,d)$-coloring and a $(c,d^{\prime})$-coloring exist, then for any $(c,d)$-coloring,
a sequence of recolorings exists that ends with some $(c,d^{\prime})$-coloring, without ever recoloring $u$, and only recoloring $v$ at the last step.
\end{itemize}
See \figurename~\ref{fig:forbidding_path}.
In the rest of the paper, a list assignment $L$ of a graph $G$ satisfies $L(v) \subseteq \{0,1,2,3\}$ for $v \in V(G)$.
Bonsma and Cereceda~\cite[Lemma 7]{bonsma} proved that an $(a,b)$-forbidding $(u,v)$-path exists
if $L(u) \neq \{0,1,2,3\}$ and $L(v) \neq \{0,1,2,3\}$.
We need an $(a,b)$-forbidding $(u,v)$-path satisfying additional conditions.

\begin{figure}
    \centering
    \includegraphics{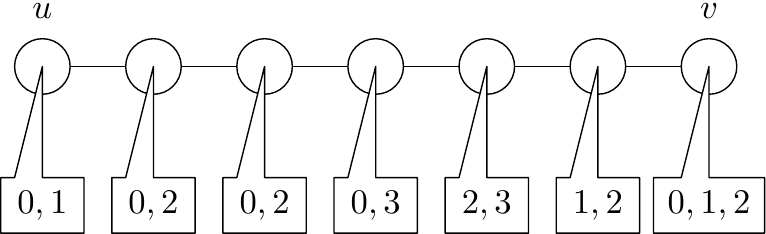}
    \caption{An example of a $(0,1)$-forbidding $(u,v)$-path. Each balloon shows the list of colors associated to each vertex. This example is constructed by the method in the proof of \Cref{forbidding} with $c=3$.}
    \label{fig:forbidding_path}
\end{figure}

\begin{lemma}\label{forbidding}
Let $L_u \subsetneq \{0,1,2,3\}$ and $L_v \subsetneq \{0,1,2,3\}$ with $L_u \cup L_v \neq \{0,1,2,3\}$.
For any $a \in L_u$, $b \in L_v$, and $c \notin L_u\cup L_v$,
there exists an $(a,b)$-forbidding $(u,v)$-path $(P,L)$ satisfying that
\begin{itemize}
\item[{\rm (i)}] $L(u)=L_u$ and $L(v)=L_v$,
\item[{\rm (ii)}] $L(w) \subseteq \{0,1,2,3\}$ and $\#L(w)=2$ for each $w \in V(P)\setminus \{u,v\}$,
\item[{\rm (iii)}] $P$ has even length,
\item[{\rm (iv)}] $\bigcup_{w \in V(P)}L(w)=\{0,1,2,3\}$, and
\item[{\rm (v)}] $L(w)$ contains $c$ for each $w \in N_P(u) \cup N_P(v)$.
\end{itemize}
\end{lemma}
\begin{proof}
Let $c \in \{0,1,2,3\}\setminus (L(u) \cup L(v))$. Let $P := uv_1v_2v_3v_4v_5v$ be a path with length six.
If $a \neq b$, then let $L$ be a list assignment with
\[ (L(u),L(v_1),L(v_2),L(v_3),L(v_4),L(v_5),L(v)) :=(L_u,\{a,c\},\{a,c\},\{a,d\},\{c,d\},\{b,c\},L_v), \]
where $d$ is the color in $\{0,1,2,3\}\setminus \{a,b,c\}$.
If $a=b$, then let $L$ be a list assignment with
\[ (L(u),L(v_1),L(v_2),L(v_3),L(v_4),L(v_5),L(v)) :=(L_u,\{a,c\},\{c,d\},\{d,e\},\{c,e\},\{a,c\},L_v), \]
where $d$ and $e$ are the different colors in $\{0,1,2,3\}\setminus \{a,c\}$.

We prove that $(P,L)$ satisfies the condition (I) of an $(a,b)$-forbidding path.
Let $\alpha$ be an $L$-coloring with $\alpha(u)=a$.
If $a \neq b$, then $\alpha(v_1)=c$, $\alpha(v_2)=a$, $\alpha(v_3)=d$, $\alpha(v_4)=c$, and $\alpha(v_5)=b$;
otherwise, $\alpha(v_1)=c$, $\alpha(v_2)=d$, $\alpha(v_3)=e$, $\alpha(v_4)=c$, and $\alpha(v_5)=a$.
In either case, we have $\alpha(v) \neq b$.
We show that for any pair $(x,y) \neq (a,b)$ with $x \in L(u)$ and $y \in L(v)$, there is an $(x,y)$-coloring.
Let $(x,y) \neq (a,b)$ be a pair with $x \in L(u)$ and $y \in L(v)$.
If $x=a$, then the coloring $\alpha$ above can be extended to an $(x,y)$-coloring by letting $\alpha(v) = y$.
By a similar argument, we can conclude the case $y=b$.
If $x\neq a$ and $y \neq b$,
then there is an $(x,y)$-coloring
such that $v_1$ and $v_5$ have the color $a$ and $b$, respectively.

We show that $(P,L)$ satisfies the condition (II) of an $(a,b)$-forbidding path.
We may assume that both an $(x,y)$-coloring and an $(x^{\prime},y)$-coloring exist.
Let $\beta$ be an $(x,y)$-coloring.
If $x^{\prime} \neq a$, then by $c \notin L(u)\cup L(v)$,
we can recolor $\beta(u)$ to $x^{\prime}$ in the case either $\beta(v_1) = a$ or $\beta(v_1) = c$.
Hence, we may assume that $x^{\prime} = a$. Then, we have $y \neq b$.
If $\beta(v_1) \neq a$, then we can recolor $\beta(u)$ to $a$. Hence, we suppose that $\beta(v_1)=a$.
Let $1\leq i\leq 5$ be the smallest integer that satisfies $L(v_i)\setminus \{\beta(v_i)\} \neq \{\beta(v_{i+1})\}$, where $v_6 :=v$.
Since $y\neq b$, such an integer $i$ exists.
Then, we can recolor $\beta(v_j)$ to the color in $L(v_j)\setminus \{\beta(v_j)\}$
for $1 \leq j \leq i$ in the order from $v_i$ to $v_1$.
Hence, we can recolor $\beta(u)$ to $a$ and $(P,L)$ satisfies the condition (II) of an $(a,b)$-forbidding path.
By a similar argument, we can show that $(P,L)$ satisfies the condition (III) of an $(a,b)$-forbidding path.

It is easy to see that $(P,L)$ satisfies the conditions (i)--(v) in Lemma~\ref{forbidding}.
\end{proof}

Bonsma and Cereceda~\cite{bonsma} proved that \Lrecolor is PSPACE-complete
for particularly restricted graphs and list assignments.
To explain the restriction,
we first define the set $\mathcal{X}$ of graphs $X$ satisfying that
\begin{itemize}
\item $X$ contains mutually disjoint 
complete graphs $T_1,T_2,\ldots,T_{\ell}$ and $S_1,S_2,\ldots,S_m$ as a subgraph,
where $\# V(T_i) = 3$ for $1\leq i\leq \ell$
and $\# V(S_j) = 2$ for $1 \leq j \leq m$,
\item $V(X)=\bigcup_{1\leq i\leq \ell}V(T_i) \cup \bigcup_{1\leq j\leq m}V(S_j)$,
\item every vertex in $X$ is of degree two or three and for any $1\leq i\leq \ell$, every vertex in $T_i$ is of degree three in $X$, and
\item $X$ has a planar embedding such that each $T_i$ bounds a face.
\end{itemize}
See \figurename~\ref{fig:reduction2} (Left).

\begin{figure}
    \centering
    \includegraphics[width=0.95\textwidth]{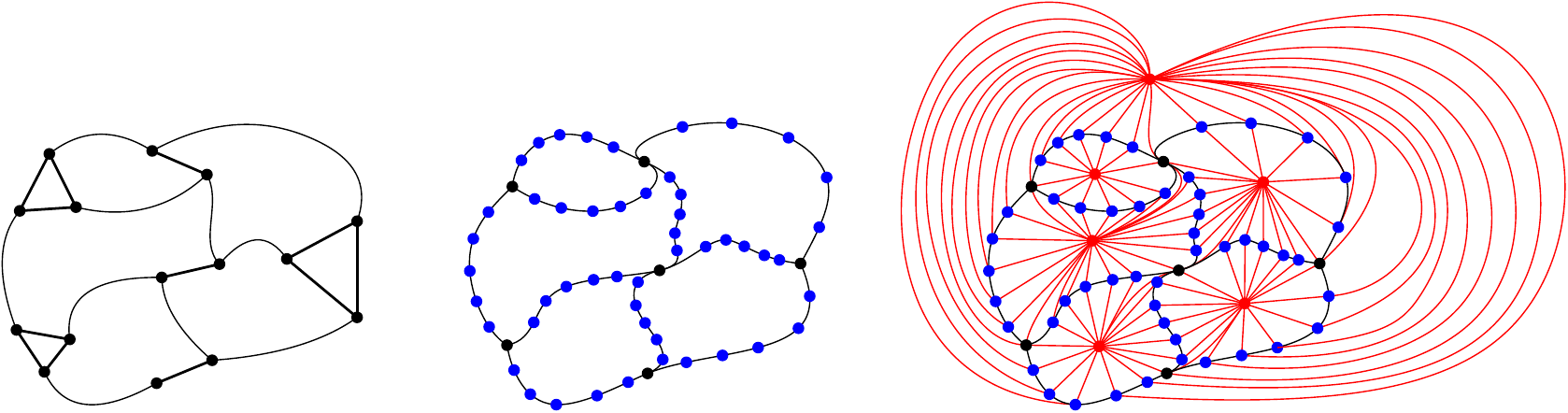}
    \caption{Reduction for the proof of \Cref{thm:pspace}. (Left) A graph $X$ by Bonsma and Cereceda~\cite{bonsma}. (Middle) The graph $H(X)$ obtained from $G$. Black vertices correspond to the contracted triangles and edges. Blue vertices come from forbidding paths. (Right) The graph $G^{\prime}$ obtained from $H(X)$. Red vertices are inserted to the faces of $H(X)$.}
    \label{fig:reduction2}
\end{figure}

Let $X$ be a graph in $\mathcal{X}$.
We give labels $t^0_i,t^1_i,t^2_i$ to the vertices of $T_i$ for $1\leq i\leq \ell$ and labels $s^0_j,s^1_j$ to the vertices in $S_j$ for $1\leq j\leq m$.
Let $X^{\prime}$ be the graph obtained from $X$ by contracting $T_i$ into a vertex $t_i$ for $1\leq i\leq \ell$ and contracting $S_j$ into a vertex $s_j$ for $1\leq j\leq m$.
Let $H(X)$ be the graph obtained from $X^{\prime}$ by replacing
each edge of $X^{\prime}$ with the form $\{t_i, s_j\}$, $\{t_i, t_{i^{\prime}}\}$, or $\{s_j,s_{j^{\prime}}\}$
by an $(a,b)$-forbidding path satisfying the conditions of Lemma~\ref{forbidding}
letting $a$ and $b$ be the integers satisfying $\{t_i^a, s_j^b\},\{t_i^a, t_{i^{\prime}}^b\},\{s_j^a,s_{j^{\prime}}^b\} \in E(X)$,
respectively.
See \figurename~\ref{fig:reduction2} (Middle).
For a vertex $v \in V(X^{\prime})$, the vertices adjacent to $v$ in $X^{\prime}$ are called the \emph{pseudo-neighbors} of $v$ in $H(X)$.
We denote by $\mathcal{H}$ the set of graphs $H(X)$ for $X \in \mathcal{X}$. 
Let $L_H$ be the list assignment of $H$ with $L_H(t_i)=\{0,1,2\}$ for $1\leq i\leq \ell$ and $L_H(s_j)=\{0,1\}$ for $1\leq j\leq m$.

\begin{theorem}[\cite{bonsma}]\label{lrecolor}
\Lrecolor is PSPACE-complete for $H \in \mathcal{H}$ and the list assignment $L_H$.
\end{theorem}

We prove the following lemma.

\begin{lemma}\label{2-conn}
Let $(H, L_H, \alpha,\beta)$ be an instance of \Lrecolor with $H \in \mathcal{H}$.
Then there is an instance $(H^{\prime},L,\alpha^{\prime},\beta^{\prime})$ of \Lrecolor
such that $H^{\prime}$ is a $2$-connected planar graph and $(H, L_H, \alpha,\beta)$ is a YES-instance if and only if $(H^{\prime},L,\alpha^{\prime},\beta^{\prime})$ is a YES-instance.
\end{lemma}

Let $G$ be a connected graph.
For $v \in V(G)$,
we denote by $G - v$ the graph obtained from $G$ by removing $v$ (and all incident edges to $v$).
A vertex $v \in V(G)$ is called a \emph{cut vertex}
if $G-v$ is not connected.
A \emph{block} of $G$ is a maximal connected subgraph of $G$ without any cut vertex.

\begin{proof}[Proof of Lemma~\ref{2-conn}]
We may assume that $H$ is connected.
Let $X$ be a graph in $\mathcal{X}$ such that $H=H(X)$.
Let $H_0:=H$ and for $i\geq 0$,
suppose that $H_i$ contains a cut vertex $v$ of $H_i$ with $v \in V(X^{\prime})$.
Then there are two pseudo-neighbors $x$ and $y$ of $v$
such that $x$ and $y$ are contained in different components of $H_i-v$
and lie on the boundary of a face $f$ of $H_i$.
Let $H_{i+1}$ be the plane graph obtained from $H_i$ by
connecting $x$ and $y$
by a path $xu_iy$ so that
it divides the face $f$ into two new faces.
Note that the number of blocks of $H_{i+1}$ is strictly less than the number of blocks of $H_i$.
Thus, performing this operation as long as possible,
we can find an integer $\ell$ such that $H_{\ell}$ has no cut vertex $v$ with $v \in V(H)$.
It is easy to see that $H_{\ell}$ is indeed $2$-connected.

Let $L$ be the list assignment of $H_{\ell}$ defined by
\[
L(v):=\begin{cases}
L_H(v) &\text{if $v \in V(H)$,} \\
\{2,3\} &\text{if $v \in \{u_i \mid  1\leq i\leq \ell\}$.}
\end{cases}
\]
For a map $\gamma$ from $V(G)$ to $\{0,1,2,3\}$,
let $\gamma^{\prime}$ be defined as
\begin{align*}
\gamma^{\prime}(v) &:=
\begin{cases}
\gamma(v) &\text{if $v \in V(H)$,} \\
3 &\text{if $v \in \{u_i \mid  1\leq i\leq \ell\}$}.
\end{cases}
\end{align*}
Since the list of a vertex in $\bigcup_{1\leq i\leq \ell}\left(N_{H_{\ell}}(u_i)\right)$ does not contain $3$,
a map $\gamma$ from $V(G)$ to $\{0,1,2,3\}$ is an $L_H$-coloring of $H$ if and only if
$\gamma^{\prime}$ is an $L$-coloring of $H_{\ell}$.
Hence $(H,L_H,\alpha,\beta)$ is a YES-instance if and only if $(H_{\ell},L,\alpha^{\prime},\beta^{\prime})$ is a YES-instance.
This completes the proof of Lemma~\ref{2-conn}.
\end{proof}

We denote by $\mathcal{H}^{\prime}$
the set of graphs $H^{\prime}$ defined as in Lemma~\ref{2-conn} for $H \in \mathcal{H}$.
For $H \in \mathcal{H}^{\prime}$, let $L_H$ be the list assignment defined as $L$ in the proof of Lemma~\ref{2-conn}.
By the condition of Lemma~\ref{forbidding} and the definition of $L_H$, we obtain the following lemma, which is used in the proof of Theorem~\ref{thm:equivalent}.

\begin{lemma}\label{useful}
Let $H$ be a graph in $\mathcal{H}^{\prime}$.
For a vertex $v \in V(H)$ with $\#L(v)=3$, the list size of a neighbor of $v$ is two and its list contains $3$.
\end{lemma}

By Theorem \ref{lrecolor} and Lemma~\ref{2-conn}, we obtain the following theorem.

\begin{theorem}\label{thm:2-conn}
\Lrecolor is PSPACE-complete for $H \in \mathcal{H}^{\prime}$ and the list assignment $L_H$.
\end{theorem}

\subsection{Proof of Theorem~\ref{thm:pspace}}

The following theorem is crucial in the proof of \Cref{thm:pspace}:
\begin{theorem}\label{thm:equivalent}
Let $(H, L_H, \alpha,\beta)$ be an instance of \Lrecolor
with $H \in \mathcal{H}^{\prime}$ defined as in Section~\ref{pre}.
Then, for every $k\geq 4$,
there is an instance $(G, \alpha^{\prime}, \beta^{\prime})$ 
of \kplusrecolor,
where
$G$ is a $(k-1)$-colorable triangulation of the $(k-2)$-sphere,
such that
$(H,L_H,\alpha,\beta)$ is a YES-instance if and only if $(G,\alpha^{\prime}, \beta^{\prime})$ is a YES-instance.
\end{theorem}

If \Cref{thm:equivalent} holds,
we can obtain \Cref{thm:pspace} as follows.
\begin{proof}[Proof of Theorem~\ref{thm:pspace}]
Let $(H, L_H, \alpha, \beta)$ be an instance of \Lrecolor
with $H \in \mathcal{H}$,
defined as in Section~\ref{pre}.
By Theorem~\ref{thm:equivalent},
we can construct in polynomial time
an instance $(G_k, \alpha^{\prime}, \beta^{\prime})$
of \kplusrecolor,
where $G_k$ is a $(k-1)$-colorable triangulation of $(k-2)$-sphere,
such that
$(H,L_H,\alpha,\beta)$ is a YES-instance of \Lrecolor if and only if $(G_k,\alpha^{\prime}, \beta^{\prime})$ is a YES-instance of \kplusrecolor.
This together with Theorem~\ref{thm:2-conn} show that
\kplusrecolor for $(k-1)$-colorable triangulations of the $(k-2)$-sphere is PSPACE-complete.
\end{proof}

In the following, we present the proof of \Cref{thm:equivalent}.
For a graph $G$ and a vertex $v \in V(G)$,
let $\deg_G(v)$ denote the number of edges incident to $v$,
i.e., $\deg_G(v) := \# \delta_G(v)$.
Recall that
a triangulation $K$ of a simply-connected closed $d$-manifold is $(d+1)$-colorable if and only if it is even,
i.e., $\#\St^{d}(\sigma^{d-2})$ is even for every $(d-2)$-simplex $\sigma^{d-2}\in K$.

\begin{proof}[Proof of \Cref{thm:equivalent}]
We prove Theorem~\ref{thm:equivalent} by induction on $k \geq 4$.
We first show the case of $k=4$.
We construct a $3$-colorable triangulation $G$ of the $2$-sphere from $H$.
For each face $f$ of $H$, we add a new vertex $v_f$ to $f$ and add an edge connecting $v_f$ and every vertex lying on the boundary of $f$.
Let $G^{\prime}$ be the resulting graph. Then, every face of $G^{\prime}$ contains a vertex not in $V(H)$.
See \figurename~\ref{fig:reduction2} (Right).

\begin{claim}\label{even_tri1}
$G^{\prime}$ is a $3$-colorable triangulation of the $2$-sphere.
\end{claim}
\begin{proof}
Since $H$ is $2$-connected, there is no multiple edges in $G^{\prime}$.
Thus, it is easy to see that $G^{\prime}$ is a triangulation.
We only have to prove that the degree of every vertex of $G^{\prime}$ is even.
Let $v$ be a vertex of $G^{\prime}$. If $v$ is contained in $V(H)$, then $\deg_{G^{\prime}}(v)=\# \St^2_{H}(v)+\deg_{H}(v)=2\deg_{H}(v)$.
If $v$ is not contained in $V(H)$, then $\deg_{G^{\prime}}(v)$ is equal to the number of vertices lying on the boundary of the face of $H$ containing $v$.
Since every forbidding path in $H$ has even length,
the number of vertices lying on the boundary of each face of $H$ is even.
In either case, $\deg_{G^{\prime}}(v)$ is even.
\end{proof}

\begin{figure}
\centering
\includegraphics{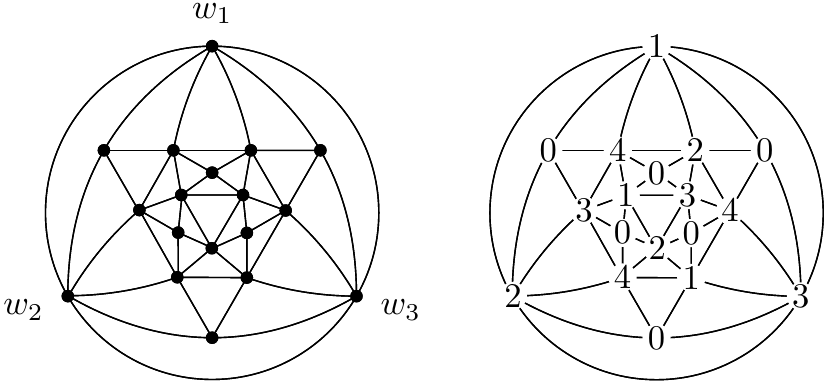}
\caption{(Left) The graph $J$. (Right) A frozen $5$-coloring of $J$.}\label{frozen}
\end{figure}

Let $J$ be the $3$-colorable triangulation of the $2$-sphere shown in \figurename~\ref{frozen} (Left) and
let $w_1$, $w_2$, and $w_3$ be the vertices lying on the boundary of the outer face of $J$.
There is a $5$-coloring $c$ of $J$ such that $\bigcup_{w \in N_J(v)}c(w)=\{0,1,2,3,4\}\setminus \{c(v)\}$ for every vertex $v \in V(J)$, i.e.,
all colors except for $c(v)$ appear in the neighbors of $v$; see \figurename~\ref{frozen} (Right).
Such a $5$-coloring is said to be \emph{frozen} since no single-change can be performed.
We refer to a frozen $5$-coloring $\alpha$ as a \emph{$(c_1,c_2,c_3)$-frozen $5$-coloring} if $\alpha(w_i)=c_i$ for $1\leq i\leq 3$.
Let $G$ be the plane graph obtained from $G^{\prime}$ as follows:
For each face $f=\{x^f,y^f,z^f\}$ of $G^{\prime}$ with $x^f \notin V(H)$, we add the graph $J_f$,
which is isomorphic to $J$, to $f$ and edges $w_1^fx^f,w_1^fy^f,w_2^fx^f,w_2^fz^f,w_3^fy^f,w_3^fz^f$
such that $\{w_1^f,w_2^f,x^f\}, \{w_1^f, w_3^f, y^f\},\{w_2^f, w_3^f, z^f\}$ are faces of $G$, where $w_i^f$ is the vertex in $J_f$ corresponding to $w_i$ for $1\leq i\leq 3$.
See \figurename~\ref{fig:inserting_J1}.

\begin{figure}
    \centering
    \includegraphics{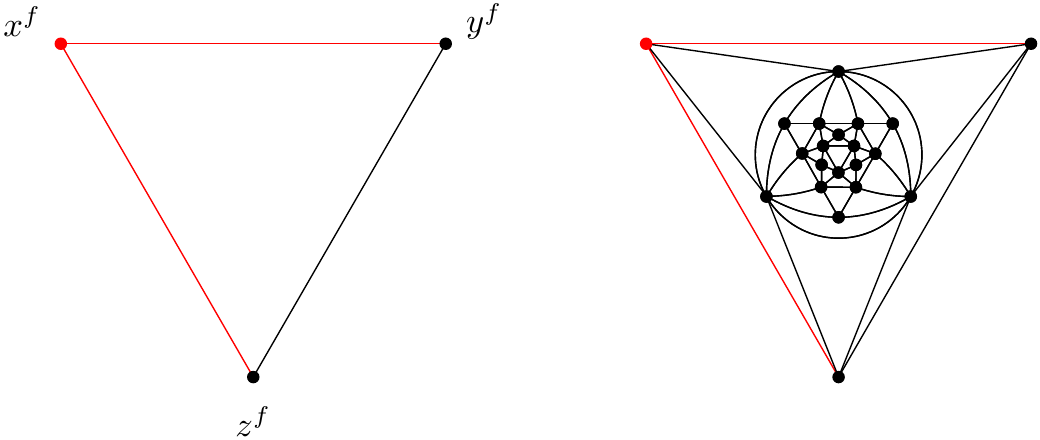}
    \caption{Inserting a copy of $J$ in each face $f$ of $G'$ to obtain the graph $G$.}
    \label{fig:inserting_J1}
\end{figure}

\begin{claim}\label{even_tri2}
$G$ is a $3$-colorable triangulation of the $2$-sphere.
\end{claim}
\begin{proof}
It is easy to see that $G$ is a triangulation of the $2$-sphere.
We only have to prove that the degree of every vertex of $G$ is even.
Let $v$ be a vertex of $G$. If $v$ is contained in $V(G^{\prime})$, then $\deg_G(v)=\deg_{G^{\prime}}(v)+2 \cdot \# \St^2_{G^{\prime}}(v)$ is even since $\deg_{G^{\prime}}(v)$ is even.
If $v$ is not contained in $V(G^{\prime})$, then $v$ is contained in $J_f$ for some face $f$ of $G^{\prime}$ and
\[ \deg_G(v)=
\begin{cases}
\deg_{J_f}(v) & \text{if $v \notin \{w_1^f,w_2^f,w_3^f\}$,} \\
\deg_{J_f}(v)+2 & \text{if $v \in \{w_1^f,w_2^f,w_3^f\}$,}
\end{cases}
\]
which is even since $\deg_{J_f}(v)$ is even.
Hence, $G$ is a $3$-colorable triangulation of the $2$-sphere.
\end{proof}

For an $L_H$-coloring $\alpha$ of $H$, a $5$-coloring $\alpha^{\prime}$ of $G$ is said to be \emph{restricted} with respect to $\alpha$
if $\alpha^{\prime}$ satisfies the following conditions:
\begin{enumerate}
\item[(a)] $\alpha^{\prime}(v)=\alpha(v)$ for every $v \in V(H)$,
\item[(b)] $\alpha^{\prime}(v)=4$ for each vertex $v \in \{w_3^f \mid f \in F(G^{\prime})\} \cup V(G^{\prime})\setminus V(H)$,
\item[(c)] 
$\alpha^{\prime}(w_1^f) \in \{0,1,2,3\} \setminus L_H(y^f)$ and
$\alpha^{\prime}(w_2^f) \in \{0,1,2,3\} \setminus L_H(z^f)$ for every face $f$ of $G^{\prime}$, and
\item[(d)] $L_H(v)=\{0,1,2,3,4\}\setminus \bigcup_{w \in N_G(v)\setminus V(H)}\alpha^{\prime}(w)$ for every vertex $v \in V(H)$.
\end{enumerate}
\begin{claim}
For an $L_H$-coloring $\alpha$ of $H$, there is a restricted $5$-coloring $\alpha^{\prime}$ of $G$ with respect to $\alpha$.
\end{claim}

\begin{proof}
For convenience, let $m(v):=\# \St^2_{G^{\prime}}(v)$ for $v \in V(G^{\prime})\setminus V(H)$.
For a vertex $v \in V(G^{\prime})\setminus V(H)$,
let $f_1^v,f_2^v,\ldots,f_{m(v)}^v$ be the faces in $\St^2_{G^{\prime}}(v)$ such that
$f_1^v,f_2^v,\ldots,f_{m(v)}^v$ appear on $\St^2_{G^{\prime}}(v)$ in this clockwise order and $z^{f_i^v}=y^{f_{i+1}^v}$ for $1\leq i\leq m(v)$,
where $y^{f_{m(v)+1}^v} := y^{f_1^v}$.
We prove the following subclaim.
\begin{subclaim} Let $v$ be a vertex in $V(G^{\prime})\setminus V(H)$ and let $i$ be an integer with $1\leq i\leq m(v)$.
Then the following hold.
\begin{enumerate}
    \item[{\rm (1)}] For each color $c \in \{0,1,2,3\}\setminus L_H(y^{f_i})$, there exists a color in $\{0,1,2,3\}\setminus (L_H(z^{f_i})\cup \{c\})$.
    \item[{\rm (2)}] If $\# L_H(z^{f_i^v})=3$, then for each color $c \in \{0,1,2,3\}\setminus L_H(y^{f_i})$,
    a color $3$ is contained in $\{0,1,2,3\}\setminus (L_H(z^{f_i})\cup \{c\})$.
\end{enumerate}
\end{subclaim}
\begin{proof}
(1). The construction of $H$ implies that $\# L_H(y^{f_i^v})$ and $\# L_H(z^{f_i^v}) $ are at most three and one of $\# L_H(y^{f_i^v}) $ and $\# L_H(z^{f_i^v})$ is exactly two.
Moreover, if one of $\# L_H(y^{f_i^v})$ and $\# L_H(z^{f_i^v})$ is three, then its list is $\{0,1,2\}$ and the other list contains $3$
by Lemma~\ref{useful}.
Hence for each color $c \in \{0,1,2,3\} \setminus L_H(y^{f_i^v})$,
there exists a color $c^{\prime} \in \{0,1,2,3\} \setminus (L_H(z^{f_i^v}) \cup \{c\})$.

(2) If $\# L_H(z^{f_i^v})=3$, then $L_H(z^{f_i^v})=\{0,1,2\}$ and $L_H(y^{f_i^v})$ contains $3$ and so the claim holds.
\end{proof}

We may assume that $y^{f_1^v}$ is an end-vertex of some forbidding path.
Let $\alpha_v$ be a map from $\{w_j^{f_i^v} \mid 1\leq i\leq m(v),\ 1\leq j\leq 2 \}$ to $\{0,1,2,3\}$ defined as follows:
\[ \alpha_v(w_1^{f_1^v}) :=\begin{cases}
2 &\text{if $L_H(y^{f_1^v})=\{0,1\}$}, \\
3 &\text{otherwise,}
\end{cases} \]
$\alpha_v(w_2^{f_1^v}) \in \{0,1,2,3\}\setminus (L_H(z^{f_1^v}) \cup \{\alpha_v(w_1^{f_1^v})\})$, and
for $2\leq i\leq m(v)$,
\begin{align*}
\alpha_v(w_1^{f_i^v})
&\begin{cases}
\in \{0,1,2,3\}\setminus (L_H(y^{f_i^v}) \cup \{\alpha_v(w_2^{f_{i-1}^v})\}) &\text{if $\# L_H(y^{f_i^v})=2$}, \\
:= 3 &\text{otherwise},
\end{cases}\\
\alpha_v(w_2^{f_i^v})
&\begin{cases}
\in \{0,1,2,3\}\setminus (L_H(z^{f_i^v}) \cup \{\alpha_v(w_1^{f_i^v})\}) &\text{if $\# L_H(z^{f_i^v})=2$}, \\
:= 3&\text{otherwise.}
\end{cases}
\end{align*}
By Subclaim above, there exists such a map $\alpha_v$ for every $v \in V(G^{\prime})\setminus V(H)$.
\begin{subclaim} For every vertex $v \in V(G^{\prime})\setminus V(H)$, the map $\alpha_v$ satisfies the following conditions: For every $1\leq i\leq m(v)$,
\begin{enumerate}
    \item[{\rm (i)}] $\alpha_v(w_1^{f_i^v}) \in \{0,1,2,3\} \setminus L_H(y^{f_i^v})$ and $\alpha_v(w_2^{f_i^v}) \in \{0,1,2,3\} \setminus L_H(z^{f_i^v})$,
    \item[{\rm (ii)}] $\alpha_v(w_1^{f_i^v}) \neq \alpha_v(w_2^{f_i^v})$, and
    \item[{\rm (iii)}] $L_H(y^{f_i^v})=\{0,1,2,3\}\setminus \{\alpha_v(w_2^{f_{i-1}^v}),\alpha_v(w_1^{f_i^v})\}$ and $L_H(z^{f_i^v})=\{0,1,2,3\}\setminus \{\alpha_v(w_2^{f_i^v}),\alpha_v(w_1^{f_{i+1}^v})\}$.
\end{enumerate}
\end{subclaim}
\begin{proof}
Let $i$ be an integer with $1\leq i\leq m(v)$.

We first prove (i) and (ii) together. If $\# L_H(z^{f_i^v})=2$, then (i) and (ii) hold by the definition of $\alpha_v$.
Suppose that $\# L_H(z^{f_i^v}) \neq 2$. Then $L_H(z^{f_i^v})=\{0,1,2\}$ and so (i) holds. By the construction of $H$ and Lemma~\ref{useful},
$\# L_H(y^{f_i^v})=2$ and $L_H(y^{f_i^v})$ contains $3$.
Hence, $\alpha_v(w_1^{f_i^v}) \neq 3$ and so (ii) holds.

We next prove (iii). If $\# L_H(y^{f_i^v})\neq 2$ (resp.\ $\# L_H(z^{f_i^v})\neq 2$),
then $L_H(y^{f_i^v})=\{0,1,2\}$ (resp.\ $L_H(z^{f_i^v})=\{0,1,2\}$) and $\alpha_v(w_1^{f_i^v})=3$ (resp.\  $\alpha_v(w_2^{f_i^v})=3$).
Hence (iii) holds.
If $\# L_H(y^{f_i^v})=2$ (resp.\ $\# L_H(z^{f_i^v})=2$), then $\alpha_v(w_2^{f_{i-1}^v}) \neq \alpha_v(w_1^{f_i^v})$
(resp. $\alpha_v(w_2^{f_i^v}) \neq \alpha_v(w_1^{f_{i+1}^v})$) and this together with (i) shows that (iii) holds.
\end{proof}

Let $\alpha^{\prime}$ be a map from $V(G)$ to $\{0,1,2,3,4\}$ as follows:
\begin{itemize}
\item $\alpha^{\prime}(v) = \alpha(v)$ for every $v \in V(H)$,
\item $\alpha^{\prime}(v) = 4$ for each vertex $v \in \{w_3^f \mid f \in F(G^{\prime})\} \cup V(G^{\prime})\setminus V(H)$, and
\item
$\alpha^{\prime}(x) = \alpha_{f_i^v}(x)$
for each vertex $v \in V(G^{\prime})\setminus V(H)$,
$1\leq i\leq m(v)$, 
and $x \in V(J_{f_i^v})\setminus \{w_3^{f_i^v}\}$,
where $\alpha_{f_i^v}$ is an $(\alpha_v(w_1^{f_i^v}),\alpha_v(w_2^{f_i^v}),\alpha^{\prime}(w_3^{f_i^v}))$-frozen $5$-coloring of $J_{f_i^v}$.
\end{itemize}
By Subclaim above, $\alpha^{\prime}$ is a restricted $5$-coloring with respect to $\alpha$.
\end{proof}

Let $\alpha^{\prime}$ and $\beta^{\prime}$ be restricted $5$-colorings of $G$ with respect to $\alpha$ and $\beta$, respectively.
By the condition (iv) of Lemma~\ref{forbidding}, for a vertex $v \in V(G^{\prime})\setminus V(H)$,
we have $\bigcup_{w \in N_G(v)}\alpha^{\prime}(w)=\bigcup_{w \in N_G(v)}\beta^{\prime}(w)=\{0,1,2,3\}$, i.e., 
all colors in $\{0,1,2,3\}$ appear in the neighbors of $v$ in the colorings $\alpha^{\prime}$ and $\beta^{\prime}$.
For $f \in F(G^{\prime})$, since the restrictions of $\alpha^{\prime}$ and $\beta^{\prime}$ to $J_f$ are frozen $5$-colorings,
we have $\bigcup_{w \in N_G(v)}\alpha^{\prime}(w)=\{0,1,2,3,4\}\setminus \{\alpha^{\prime}(v)\}$ and
$\bigcup_{w \in N_G(v)}\beta^{\prime}(w)=\{0,1,2,3,4\}\setminus \{\beta^{\prime}(v)\}$
for every vertex $v \in V(G)\setminus V(H)$.
Hence, $(H,L_H,\alpha,\beta)$ is a YES-instance of \Lrecolor if and only if $(G,\alpha^{\prime}$, $\beta^{\prime})$ is a YES-instance of \krecolor, and Theorem~\ref{thm:pspace} holds for $k=4$.

We prove the case $k > 4$ by the induction on $k$.
To this end, we introduce the suspension of a simplicial complex.
Let $K$ be a triangulation of the $d$-sphere.
Since the $(d+1)$-sphere is obtained by gluing two $(d+1)$-balls along their boundaries, which are the $d$-spheres,
the join of $K$ with additional two points gives a triangulation $S(K)$ of
the $(d+1)$-sphere,
which is called the \emph{suspension} of $K$ (\cite[Exercise in Section~8]{Mun84}).
Here, if $K$ is even, so is $S(K)$.
See \figurename~\ref{fig:suspension_example1}.
\begin{figure}
    \centering
    \includegraphics[scale=.5]{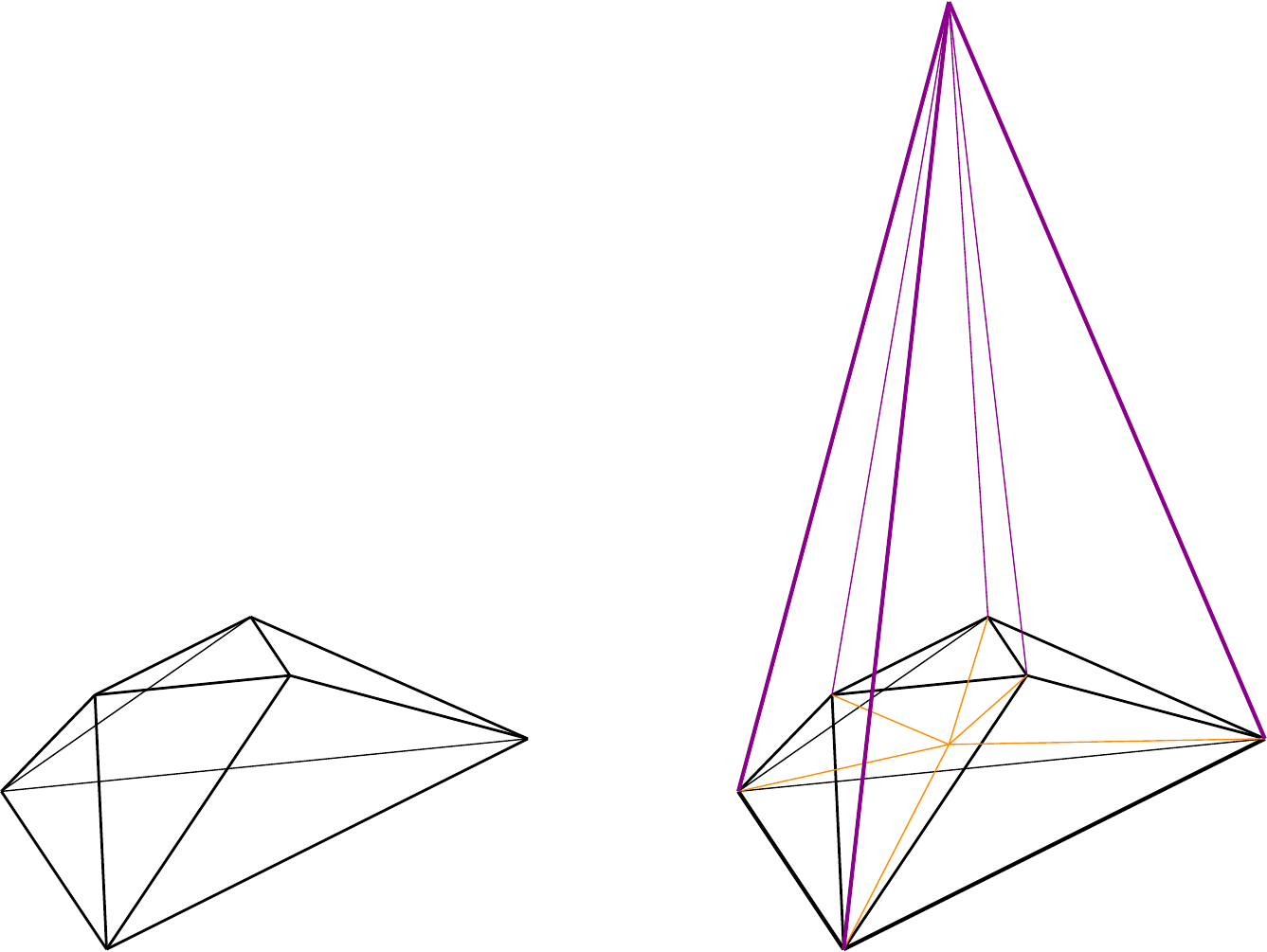}
    \caption{An example of the suspension of a triangulation of the $2$-sphere. (Left) A triangulation of the $2$-sphere. In this case, it is the boundary of an octahedron. (Right) The suspension of the left triangulation. We add a point in the interior of the sphere and connect the vertices by edges (shown in orange), and add another point in the exterior of the sphere and connect the vertices by edges (shown in purple). A $3$-simplex in the suspension is created as the join of a $2$-simplex (i.e., a face) of the triangulation and one of the added points.}
    \label{fig:suspension_example1}
\end{figure}

By the induction hypothesis, there is an instance
$(G_{k-1}, \alpha_{k-1}, \beta_{k-1})$ of \krecolor,
where $G_{k-1}$ is a $(k-2)$-colorable triangulation of the $(k-3)$-sphere,
such that
$(H,L_H,\alpha,\beta)$ is a YES-instance of \Lrecolor
if and only if 
$(G_{k-1}, \alpha_{k-1}, \beta_{k-1})$ is a YES-instance of \krecolor.
Let $G_k$ be the $1$-skeleton of the suspension of $G_{k-1}$.
Note that $G_k$ is an even triangulation of the $(k-2)$-sphere.
Let $\alpha_k$ and $\beta_k$ be the $(k+1)$-colorings of $G_k$ defined as follows:
\[
\alpha_k(v):=
\begin{cases}
\alpha_{k-1}(v) &\text{if $v \in V(G_{k-1})$},\\
k &\text{otherwise},
\end{cases} \qquad
\beta_k(v):=
\begin{cases}
\beta_{k-1}(v) &\text{if $v \in V(G_{k-1})$},\\
k &\text{otherwise}.
\end{cases}
\]
Let $x$ and $y$ be the vertices contained in $V(G_k)\setminus V(G_{k-1})$.
Then, $\bigcup_{w \in N_{G_k}(x)}\alpha_k(w)=\bigcup_{w \in N_{G_k}(x)}\beta_k(w)=\{0,1,\ldots,k-1\}$ and
$\bigcup_{w \in N_{G_k}(y)}\alpha_k(w)=\bigcup_{w \in N_{G_k}(y)}\beta_k(w)=\{0,1,\ldots,k-1\}$.
This implies that $(G_{k-1}, \alpha_{k-1}, \beta_{k-1})$ is a YES-instance
of \krecolor
if and only if $(G_k, \alpha_k, \beta_k)$ is a YES-instance
of \kplusrecolor,
and so
Theorem~\ref{thm:equivalent} holds.
\end{proof}

\section{Concluding remarks}\label{sec:conclusion}
In this paper,
we obtain the following results.
\begin{itemize}
    \item[(i)]
    For a $3$-colorable triangulation $G$ of the $2$-sphere,
    the balanced condition (B) characterizes the $3$-coloring component of $\R_4(G)$ (\Cref{thm:main}).
    More generally, for a $(k-1)$-colorable triangulation $G$ of the $(k-2)$-sphere,
    the equation~\eqref{eq:highdim B},
    which is
    the high-dimensional version of the balanced condition (B), characterizes the $(k-1)$-coloring component of $\R_k(G)$ (\Cref{thm:highdim_case}).
    \item[(ii)]
    For a $3$-colorable triangulation $G$ of the $2$-sphere,
    the $4$-coloring reconfiguration graph $\R_4(G)$ is connected if and only if every $4$-connected piece of $G$ is isomorphic to the octahedral graph (\Cref{thm:conn}).
    \item[(iii)]
    For a $(k-1)$-colorable triangulation $G$ of the $(k-2)$-sphere,
    \kplusrecolor is PSPACE-complete (\Cref{thm:pspace}).
\end{itemize}

We conclude this paper with several open problems related to each of the results (i)--(iii).

Related to (i), it is natural to ask the characterization for two $4$-colorings of a $3$-colorable triangulation $G$ of the $2$-sphere
to belong to the same connected component of $\R_4(G)$,
which is not necessarily the $3$-coloring component.
As we have seen in \Cref{subsec:2dim_case},
the equation~\eqref{eq:necessary} is a necessary condition for two $4$-colorings $\alpha, \alpha'$ to be in the same connected component.
However, it is not sufficient in general, even though it is sufficient for the $3$-coloring component.
Figure~\ref{fig:ozeki} illustrates two $4$-colorings that satisfy the equation~\eqref{eq:necessary} but do not belong to the same component.
\begin{figure}
\centering
\includegraphics[scale=0.7]{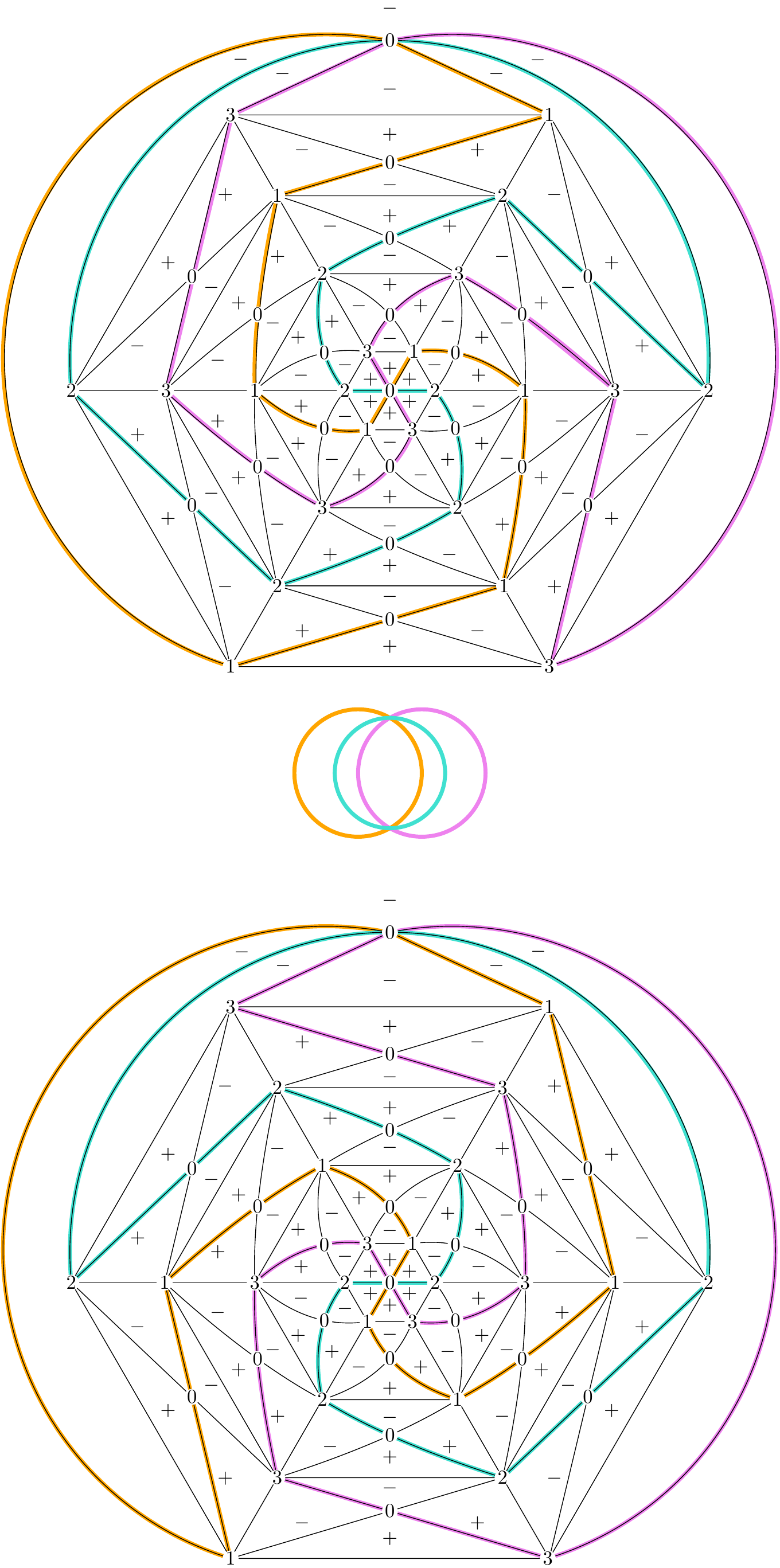}
\caption{An example showing that the equation~\eqref{eq:necessary} cannot be a characterization for two $4$-colorings of a $3$-colorable triangulation $G$ of the $2$-sphere
to belong to the same connected component of $\R_4(G)$.}
\label{fig:ozeki}
\end{figure}
We can obtain a more sophisticated necessary condition such that the 
arrangement of the Kempe-chains consisting of nonsingular edges are ``topologically equivalent'' in some sense,
but Figure~\ref{fig:ozeki} also tells us that it is still not sufficient.

The reason why such a ``topological condition'' cannot be a characterization
is that a graph may not be fine enough to approximate an ``isotopy'' of the graph in $S^2$, where an isotopy is a continuous deformation used in topology.
For example, no vertex is recolorable in the graphs in Figure~\ref{fig:ozeki}.
In order to obtain a necessary and sufficient condition,
we are required to find a ``combinatorial condition'' in addition to a ``topological condition,''
which is an interesting open problem.

A high-dimensional generalization of (ii) is also a natural problem:
In what $(k-1)$-colorable triangulation $G$ of the $(k-2)$-sphere,
$\R_k(G)$ is connected?
In our case ($k = 4$), the generating theorem (\Cref{thm:Matsu}) for $4$-connected $3$-colorable triangulation of the $2$-sphere by Matsumoto and Nakamoto~\cite{Matsumoto}
plays an important role in our proof.
Unfortunately, a high-dimensional generalization of the generating theorem is not known, 
which could be an interesting research question in its own right.

It is also natural to consider a characterization of $3$-colorable triangulations $G$ of the $2$-sphere such that $\mathcal{R}_k(G)$ is connected
and the computational complexity of \conn with $5\leq k\leq 6$,
where we recall that $\mathcal{R}_k(G)$ is always connected if $k \geq 7$~\cite{CHJ}.

We do not know the computational complexity of \sixrecolor for $3$-colorable triangulations of the $2$-sphere,
which is an open problem related to the result (iii).
We would expect the existence of a $3$-colorable triangulation $J'$ of the $2$-sphere that has a frozen $6$-coloring.
If one exists, we are able to show that \sixrecolor for $3$-colorable triangulations of the $2$-sphere is PSPACE-complete by an argument similar to the proof of \Cref{thm:equivalent} by replaceing $J$ with $J'$.
Moreover, if \sixrecolor for $3$-colorable triangulations of the $2$-sphere is PSPACE-complete,
then we are able to show that for $k\geq 4$, the problem \textsc{$(k+2)$-Recoloring} for $(k-1)$-colorable triangulation of the $(k-2)$-sphere is PSPACE-complete.

\section*{Acknowledgment}
This work was supported by JSPS KAKENHI Grant Numbers JP20H05793, JP20H05795, JP20K11670, JP20K14317, JP20K23323, JP22K17854, JP22K13956.

\appendix
\section{Proof of \Cref{lemma:small_tri}}\label{sec:ap:proof}

For simplicity of the argument,
we call a $4$-coloring that is not unbalanced \emph{balanced}.

For (1), see \figurename~\ref{fig:proof_small_tri_i}.
Since these two $4$-colorings exhaust all the cases and both of them are balanced, the octahedral graph admits no unbalanced $4$-coloring.

\begin{figure}
    \centering
    \includegraphics{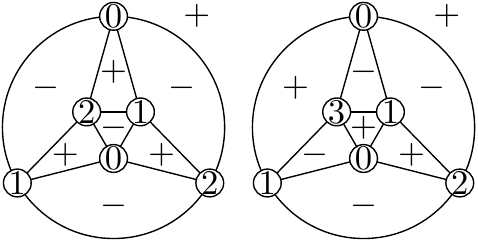}
    \caption{Proof of \Cref{lemma:small_tri}~(1): Every $4$-coloring of the octahedral graph is balanced.}
    \label{fig:proof_small_tri_i}
\end{figure}

For (2), see \figurename~\ref{fig:proof_small_tri_ii}.
There are two nonisomorphic triangulations from which the octahedral graph is obtained by a $4$-contraction.
The left one is not $3$-colorable since there are odd-degree vertices.
The right one is isomorphic to the double wheel of order $8$, where a cycle of length $6$ is highlighted with a color.

\begin{figure}
    \centering
    \includegraphics[width=.9\textwidth]{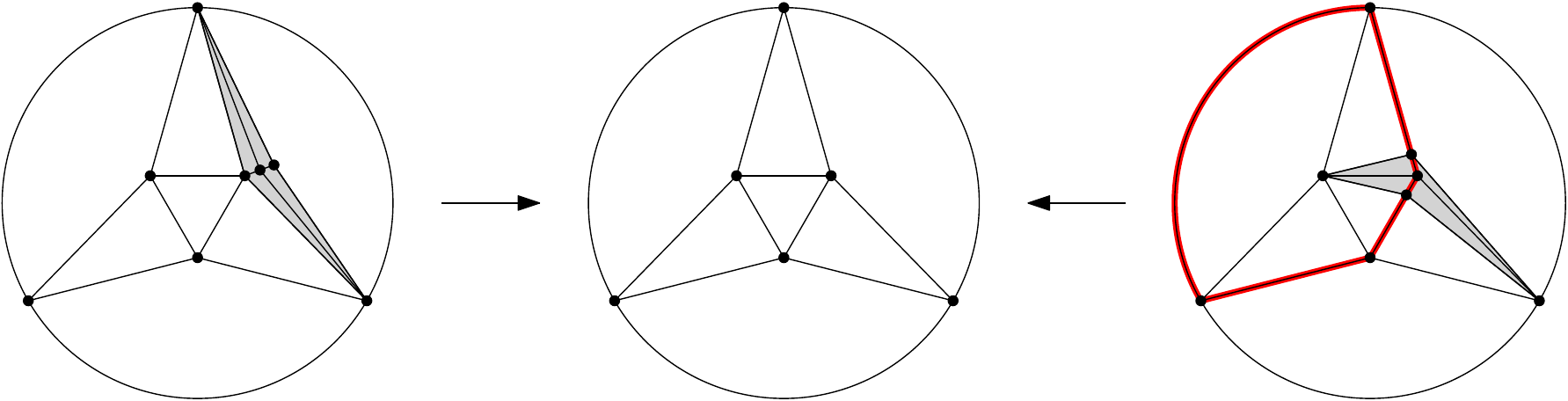}
    \caption{Proof of \Cref{lemma:small_tri}~(2): Two triangulations that result in the octahedral graph after a $4$-contraction.}
    \label{fig:proof_small_tri_ii}
\end{figure}

For~(3), see \figurename~\ref{fig:proof_small_tri_iii}.
This shows an unbalanced $4$-coloring of the double wheel of order $8$.

\begin{figure}
    \centering
    \includegraphics{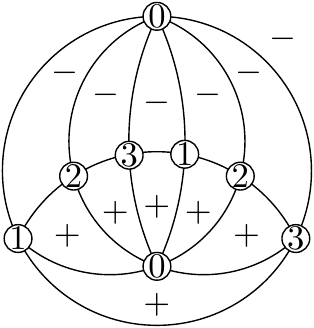}
    \caption{Proof of \Cref{lemma:small_tri}~(3): The double wheel of order $8$ with an unbalanced $4$-coloring.}
    \label{fig:proof_small_tri_iii}
\end{figure}

For~(4), see \figurename~\ref{fig:proof_small_tri_iv}.
There are two nonisomorphic triangulations from which the octahedral graph is obtained by a twin-contraction.
The left one is not $3$-colorable since there are odd-degree vertices.
The right one contains a separating triangle as highlighted with a color.

\begin{figure}
    \centering
    \includegraphics[width=0.9\textwidth]{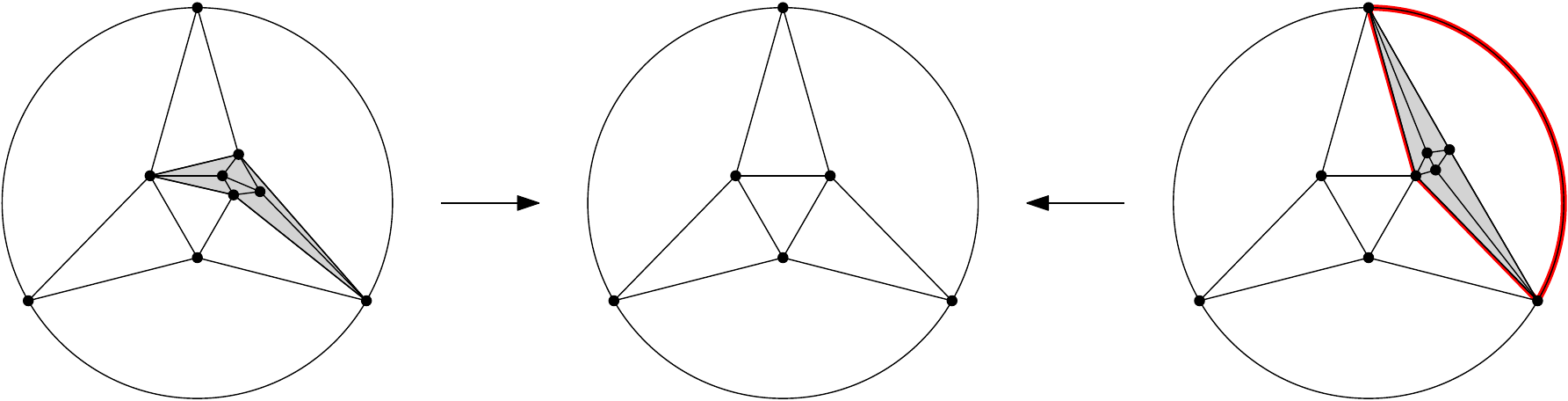}
    \caption{Proof of \Cref{lemma:small_tri}~(4): Two triangulations that result in the octahedral graph after a twin-contraction.}
    \label{fig:proof_small_tri_iv}
\end{figure}

\section{Brief review of homology groups}
\label{sec:homology}
For the convenience of the reader, we review the definition of homology groups with $\Z/2\Z$-coefficients.
See \cite[Section~10]{Mun84} for more details.
Let $K$ be a simplicial complex and let $C_q(K)$ denote the $\Z/2\Z$-vector space generated by the $q$-simplices in $K$.
For a $q$-simplex $\sigma^q=[v_0v_1\cdots v_{q}]\in K$, define $\partial\sigma^q \in C_{q-1}(K)$ by
\[
\partial\sigma^q := \sum_{i=0}^q [v_0\cdots v_{i-1} v_{i+1}\cdots v_q],
\]
and extend it to a linear map $\partial_q\colon C_{q}(K)\to C_{q-1}(K)$.
Note that we need suitable signs in the definition of $\partial\sigma^q$ in the case of $\Z$-coefficients.
One can check that $\partial_q\circ\partial_{q+1}$ is the zero map, namely $\Im\partial_{q+1} \subseteq \Ker\partial_q$.
We now define the \emph{$q$th homology group} $H_q(K;\Z/2\Z)$ of $K$ with $\Z/2\Z$-coefficients by $H_q(K;\Z/2\Z):=\Ker\partial_q/\Im\partial_{q+1}$.
Here, elements in $C_q$, $\Ker\partial_q$, and $\Im\partial_{q+1}$ are called \emph{$q$-chains}, \emph{$q$-cycles}, and \emph{$q$-boundaries}, respectively.

Let $M$ be a manifold with triangulation $K$.
We define the \emph{$q$th homology group} of $M$ by $H_q(M;\Z/2\Z):=H_q(K;\Z/2\Z)$.
This is well-defined since $H_q(K;\Z/2\Z)$ is known to be canonically isomorphic to $H_q(K';\Z/2\Z)$ for another triangulation $K'$.
For instance, $H_q(S^d;\Z/2\Z)\cong \Z/2\Z$ if $q=0,d$ and $H_q(S^d;\Z/2\Z)=\{0\}$ otherwise.

Let $M$ be a closed $d$-manifold and let $C$ be a set of $(d-1)$-simplices.
Note that $C$ can also be regarded as a $(d-1)$-chain.
Then, $C$ is a $(d-1)$-boundary if and only if $M\setminus|C|$ admits a checkerboard coloring.
Indeed, if $C \in \Im\partial_q$, then there is $x \in C_q(M)$ satisfying $\partial_q(x)=C$, and thus we obtain a checkerboard coloring by assigning black to the $d$-simplices corresponding to $x$.
Conversely, when $M\setminus|C|$ admits a checkerboard coloring, we obtain a $d$-chain $x$ by collecting the black $d$-simplices, and then $\partial_q(x)=C$.

Let us exhibit a $3$-manifold satisfying the assumption of \Cref{thm:highdim_case} except $S^3$ and a $4$-colorable triangulation of the manifold as illustrated in Figure~\ref{fig:Poincare}.
\begin{figure}
\centering
\includegraphics{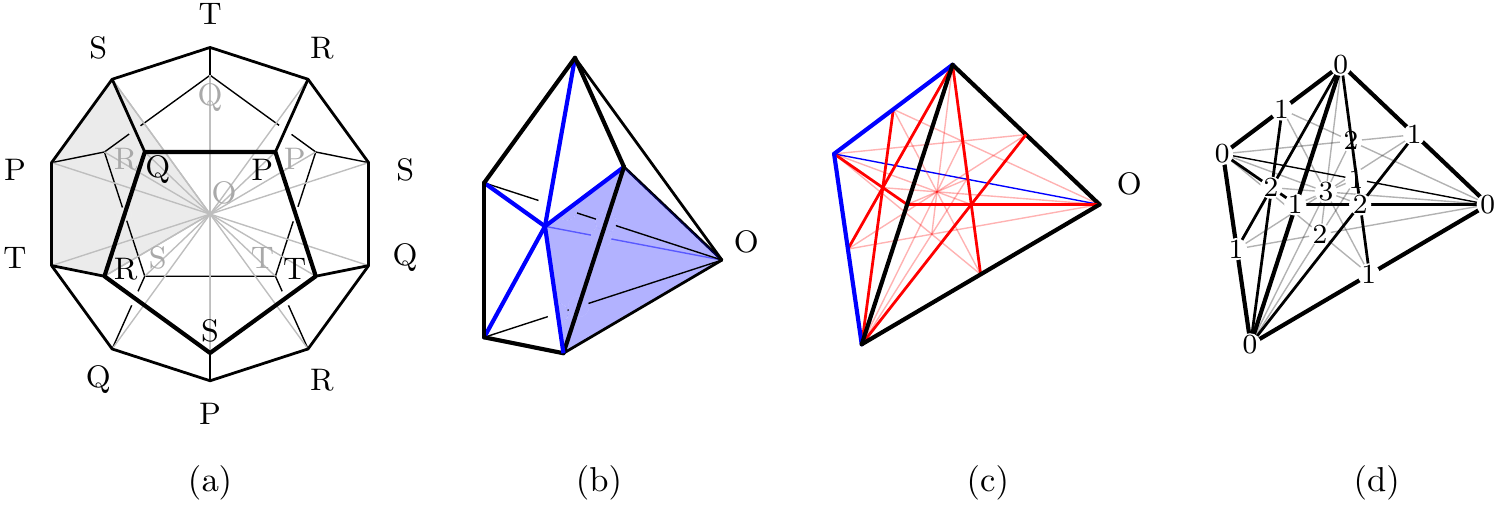}
\caption{The Poincar\'e homology $3$-sphere $P$ (\cite[p.~117]{KiSc79}). (a) To construct the Poincar\'e homology $3$-sphere, take the regular dodecahedron, and identify the antipodal pentagons with $36^{\circ}$ rotation. (b) To produce a (singular) triangulation of $P$, first subdivide each pentagonal face with one extra point in the middle, and take the join with the origin $\mathrm{O}$. (c) To make it $4$-colorable, take the barycentric subdivision. (d) We can see that this gives a $4$-colorable triangulation of $P$.}
\label{fig:Poincare}
\end{figure}

\bibliographystyle{abbrvurl}

\end{document}